\documentclass[10pt]{amsart}

\usepackage{geometry}
\usepackage[all]{xy}
\usepackage{epstopdf}
\usepackage{amssymb,amsmath,amsfonts,amsthm,graphicx,enumerate,amscd}
\usepackage{mathrsfs}
\usepackage{ytableau}
\usepackage{lscape}
\usepackage{rotating}

\newcommand{\deff}[1]{\textbf{\emph{\sharp1}}}

\newcommand{\func}[3]{\sharp1 \colon \sharp2 \to \sharp3}
\newcommand{\bb}[1]{\mathbb{\sharp1}}
\newcommand{\lie}[1]{\mathfrak{\sharp1}}
\newcommand{\iprod}[2]{\langle \sharp1, \sharp2 \rangle}
\newcommand{\ddell}[1]{\frac{\partial}{\partial \sharp1}}

\theoremstyle{plain}
\newtheorem{theorem}{Theorem}[section]

\newtheorem{claim}[theorem]{Claim}
\newtheorem{corollary}[theorem]{Corollary}
\newtheorem{lemma}[theorem]{Lemma}
\newtheorem{proposition}[theorem]{Proposition}

\theoremstyle{definition}
\newtheorem{example}[theorem]{Example}
\newtheorem{remark}[theorem]{Remark}

\newtheorem{definition}[theorem]{Definition}

\allowdisplaybreaks

\title{
Quantum multiplication through equivariant Schubert calculus}
\author{Chi-Kwong Fok}
\date{November 26, 2021}

\begin{document}
\maketitle
\begin{abstract}
	In this note, we rederive quantum Pieri's formula and the rim hook algorithm in quantum Schubert calculus by studying multiplication in the equivariant cohomology ring of Grassmannians with respect to equivariant Schubert classes which are characteristic classes. We also extend this idea in studying equivariant quantum Schubert calculus, and obtain the equivariant quantum Giambelli's and Pieri's formulae in terms of characteristic classes, with the former formula shown to be free of quantum deformation. 
\end{abstract}
\tableofcontents
\section{Introduction}
Quantum Schubert calculus is the study of the quantum cohomology ring of Grassmannians which, apart from recording information of the classical intersection of Schubert varieties, also encodes their Gromov-Witten invariants via deformation of the algebra structure of the ordinary cohomology ring. Ever since the theoretical foundation of quantum cohomology was put on a firm footing (cf. \cite{RT, KM}), there has been extensive work on the multiplication rules in quantum Schubert calculus. The quantum Pieri's and Giambelli's formulae were first obtained in \cite{Ber} by analysing Grothendieck's quot schemes. A different proof of these formulae by the observation that Gromov-Witten invariants of a Grassmannian can be realized as classical intersection numbers in a related Grassmannian was given in \cite{Bu}. In \cite{BCFF} quantum Pieri's formula and the rim hook algorithm (which is used to simplify quantum Schubert classes) were proved by purely algebraic means via Siebert-Tian's presentation (cf. \cite{ST}) of the quantum cohomology ring. By incorporating torus actions on Grassmannians, one is led to the equivariant quantum cohomology ring (cf. \cite{GK} and \cite{Ki}). Multiplication rules for this ring, in particular Pieri- and Giambelli-type formulae were proved in \cite{Mi} and \cite{Mi2} by using the notions of span and kernel in \cite{Bu} to establish the vanishing of some structure constants and introducing factorial Schur polynomials.

The aim of this note is to approach the multiplication in (equivariant) quantum Schubert calculus by way of equivariant cohomology algebraically. In the first part of the note, we recover quantum Pieri's formula and the rim hook algorithm in slightly different forms. Here we work with equivariant Schubert classes which are characteristic classes instead of canonical Schurbert classes (cohomology classes represented by the Borel homotopy quotients of Schubert varieties) which are more widely used in the literature on Schubert calculus. One advantage of characteristic classes is that they allow for a presentation of the equivariant cohomology ring which admits an easy-to-describe ring homomorphism to Siebert-Tian's presentation of the quantum cohomology ring (see Lemma \ref{homomorphism}). Tantamount to forgetting part of the equivariant structure of the equivariant cohomology, this link between equivariant and quantum Schubert calculus is the key observation of the note, as it enables us to get quantum multiplication through equivariant multiplication which, with characteristic classes, can be carried out naturally using classical multiplication rules (Pieri, Giambelli, and Littlewood-Richardson) together with a simple algorithm (Proposition \ref{ringstr}(\ref{simplify})) which simplifies equivariant Schubert classes. Along this line of thought, we prove the following version of quantum Pieri's formula.

\begin{theorem}[Quantum Pieri's formula]\label{qpieri}
	Consider the quantum product $\sigma_\lambda*\sigma_{1^r}$ in $QH^*(\text{Gr}(k, n), \mathbb{Z})$, where $\lambda$ is a Young diagram inside the $k\times(n-k)$ rectangle and $1^r$ is the Young diagram of one column with $r$ boxes, $1\leq r\leq k$. 
	\begin{enumerate}
		\item If $\lambda_1<n-k$, then $\sigma_\lambda*\sigma_{1^r}$ does not have any quantum deformation, i.e., 
		\[\sigma_\lambda*\sigma_{1^r}=\sum_{\lambda'\in\lambda\oplus 1^r}\sigma_{\lambda'}, \]
		where $\lambda\oplus 1^r$ is the set of Young diagrams in the $k\times(n-k)$ rectangle obtained by adding to $\lambda$ $r$ boxes, no two of which are in the same row.
		\item If $\lambda_1=n-k$, then 
		\[\sigma_\lambda*\sigma_{1^r}=\sum_{\lambda'\in\lambda\oplus 1^r}\sigma_{\lambda'}+q\sum_{\mu\in\lambda^-\ominus1^{k-r}}\sigma_\mu,\]
		where $\lambda^-\ominus1^{k-r}$ is the set of Young diagrams obtained by removing one box from each of any $k-r$ rows of the Young diagram $\lambda^-:=(\lambda_2, \lambda_3, \cdots, \lambda_{k})$, which is $\lambda$ with the top row deleted (here $\lambda_i$ may be zero).
	\end{enumerate}
\end{theorem}
We shall point out that the description of the quantum deformation in Theorem \ref{qpieri} is equivalent to that in \cite[Proposition 4.2]{BCFF}, but it is more specific about the number of boxes to be deleted from each of a given number of rows of $\lambda$, while the algorithm prescribed in \cite[Proposition 4.2]{BCFF} indicates the minimum number of boxes to be deleted from each column of $\lambda$. An application of the duality isomorphism of the quantum cohomology ring (Proposition \ref{qduality}, shown using the natural duality isomorphism of equivariant cohomology induced by the diffeomorphism between $\text{Gr}(k, n)$ and $\text{Gr}(n-k, n)$) leads to the `row version' of quantum Pieri's formula (Corollary \ref{rqpieri}), which is equivalent to \cite[Equation (22)]{BCFF}.

We also show a version of the rim hook algorithm (Theorem \ref{rimhook}), which can be used to simplify within the quantum cohomology ring a Schubert class $\sigma_\lambda$ with $\lambda$ not covered by the $k\times (n-k)$ rectangle to the one corresponding to a Young diagram within the rectangle. This algorithm is in fact the dual version of the one in \cite[Main Lemma]{BCFF}.

In the second part of the note, we obtain multiplication rules in the equivariant quantum cohomology ring of Grassmannians in terms of characteristic classes, which are denoted by $\widehat{\sigma}_\lambda$ with $\lambda$ being a Young diagram, and defined in Definition \ref{equivGiambelli}. We compare the bases of characteristic classes and canonical Schubert classes and use Littlewood-Richardson rule for factorial Schur polynomials which represent canonical classes to show that in the equivariant quantum setting one has Giambelli's formula without any equivariant or quantum deformation.
\begin{theorem}[Equivariant quantum Giambelli's formula]\label{eqquantgiam}
In $QH_T^*(X, \mathbb{Z})$, we have, for $\lambda\subseteq k\times(n-k)$, 
\[\widehat{\sigma}_\lambda=\det(\widehat{\sigma}_{1^{\lambda_i^T+j-i}})_{1\leq i, j\leq\ell(\lambda^T)}=\det(\widehat{\sigma}_{\lambda_i+j-i})_{1\leq i, j\leq\ell(\lambda)}.\]
\end{theorem}
We also deduce multiplication in the equivariant quantum cohomology ring through a ring homomorphism from the equivariant cohomology ring to the equivariant version of Siebert-Tian presentation for the equivariant quantum cohomology ring, extending aforementioned idea in the nonequivariant setting. 
\begin{theorem}\label{equivequivquantum}
	Let $\lambda$ and $\mu$ be Young diagrams in the $k\times(n-k)$ rectangle, and $\widetilde{f}: H_{U(n)}(\text{pt}, \mathbb{Z})\to H_{U(n)}^*(\text{pt}, \mathbb{Z})[q]$ a ring homomorphism defined by $\widetilde{f}(e_i)=e_i$ for $1\leq i\leq n-1$ and $\widetilde{f}(e_n)=e_n+(-1)^kq$. If $\widehat{\sigma}_\lambda\cdot\widehat{\sigma}_\mu=\sum_{\nu\subseteq k\times(n-k)}c_{\lambda\mu}^\nu\widehat{\sigma}_\nu$ in $H_{U(n)}^*(X, \mathbb{Z})$, then
	\[\widehat{\sigma}_\lambda*\widehat{\sigma}_\mu=\sum_{\nu\subseteq k\times(n-k)}\widetilde{c}_{\lambda\mu}^\nu\widehat{\sigma}_\nu\]
	in $QH_{U(n)}^*(X, \mathbb{Z})$, where $\widetilde{c}_{\lambda\mu}^\nu=\widetilde{f}(c_{\lambda\mu}^\nu)$. 
\end{theorem}

\begin{theorem}[Equivariant quantum Pieri's formula]\label{equivquantpieri}
	Let $\lambda$ be a Young diagram with $\ell$ rows in the $k\times(n-k)$ rectangle. Then in $QH_{U(n)}^*(\text{Gr}(k, n), \mathbb{Z})$, 
	\begin{align*}
		\widehat{\sigma}_\lambda*\widehat{\sigma}_1&=\sum_{\lambda'\in\lambda\oplus 1}\widehat{\sigma}_{\lambda'}+\delta_{\lambda_1, n-k}\left(\sum_{j=0}^{\ell-2}(-1)^{\ell-j}\sum_{m=\lambda_{\ell-j+1}}^{\lambda_{\ell-j}-1}e_{n-k+\ell-m-j}\widehat{\sigma}_{\lambda_2-1, \lambda_3-1, \cdots, \lambda_{\ell-j}-1, m, \lambda_{\ell-j+1}, \cdots, \lambda_\ell}\right.\\
		&\left.-\sum_{m=\lambda_2}^{n-k}e_{n-k-m+1}\cdot\widehat{\sigma}_{m, \lambda_2, \cdots, \lambda_\ell}+\delta_{k\ell}q\widehat{\sigma}_{\lambda_2-1, \lambda_3-1, \cdots, \lambda_\ell-1}\right),
	\end{align*}
	where $\lambda\oplus 1$ is the set of Young diagrams in the $k\times (n-k)$ rectangle obtained by adding to $\lambda$ one box.
\end{theorem}
Note that there are no `mixed terms' in the above equivariant quantum Pieri's formula, i.e., terms that involve both the equivariant and quantum variables, as is the case for equivariant quantum Pieri's formula in \cite{Mi} in terms of canonical Schubert classes.

We shall stress that, apart from the partially forgetful ring homomorphism from the equivariant cohomology to the quantum cohomology and the injective homomorphism from the equivariant cohomology to the equivariant quantum cohomology which are depicted in this note as purely algebraic results, our (re)derivations of quantum multiplication rules also crucially depend on the (equivariant) quantum Giambelli's formula (cf. \cite{Ber, Mi2}), which ultimately is built on some algebro-geometric arguments. It is natural to wonder whether these ring homomorphisms actually have any geometric connotation. In particular, we would like to see whether Gromov-Witten invariants can be interpreted as equivariant intersection of characteristic classes in a natural way, so that an alternative proof of (equivariant) quantum Giambelli's formula through these ring homomorphisms. We will explore this problem elsewhere. 

The organization of the note is as follows. In Section \ref{equivsch}, equivariant characteristic classes of Grassmannians which lift ordinary Schubert classes are defined. A presentation for the equivariant cohomology ring with these characteristic classes is given, and an algorithm for multiplying and simplifying these classes is outlined with an example. In Section \ref{qsch}, after reviewing the basics of quantum Schubert calculus and showing the ring homomorphism from equivariant to quantum cohomology, we recover quantum Pieri's formula (Theorem \ref{qpieri}) and prove a version of rim hook algorithm. We illustrate the application of the latter to multiplying general quantum Schubert classes with examples from \cite{BCFF} and contrast our computations with those in \cite{BCFF}. Section \ref{eqquantumSchCal} is devoted to the proofs of Theorems \ref{eqquantgiam}, \ref{equivequivquantum} and \ref{equivquantpieri}. We illustrate these results with an example and the multiplication table for the equivariant quantum cohomology ring of $\text{Gr}(2, 4)$. 

\textbf{Acknowledgements}: The author would like to thank Kwokwai Chan for his hospitality during a visit to the Chinese University of Hong Kong when most of the work in this note was done.
\section{Equivariant Schubert calculus via characteristic classes}\label{equivsch}
Let $T:=(S^1)^n$ be an $n$-dimensional torus. The $T$-equivariant cohomology coefficient ring $H_T^*(\text{pt}, \mathbb{Z})$ is isomorphic to the polynomial ring $\mathbb{Z}[t_1, t_2, \cdots, t_n]$. Let $e_i$ be the $i$-th elementary symmetric polynomials in $t_1, \cdots, t_n$. Let $T$ act on $\mathbb{C}^n$ by coordinatewise multiplication, and $X$ the Grassmannian $\text{Gr}(k, n)$.  The $T$-action on $\mathbb{C}^n$ induces a natural $T$-action on $X$, which is well-known to be equivariantly formal, i.e., the equivariant cohomology $H_T^*(X, \mathbb{Z})$ is isomorphic, as a $H_T^*(\text{pt}, \mathbb{Z})$-module, to $H_T^*(\text{pt}, \mathbb{Z})\otimes H^*(X, \mathbb{Z})$ (cf. \cite{GKM}). In the literature on equivariant Schubert calculus, a module basis for $H_T^*(X, \mathbb{Z})$ is often chosen to be the so-called canonical Schubert classes (cf. \cite{KT}, \cite{Mi}), which are defined as the Poincar\'e duals of the homotopy quotients of Schubert varieties determined by the standard flag of $\mathbb{C}^n$. However, in earlier papers on equivariant quantum cohomology of Grassmannians (cf. \cite{GK}, \cite{Ki}), equivariant Chern classes of tautological bundles are taken as the $H_T^*(\text{pt}, \mathbb{Z})$-algebra generators. In this section, we will describe a module basis consisting of characteristic classes which lift equivariantly the ordinary Schubert classes. 

Let $S$ be the tautological bundle of $X$ and $Q:=\underline{\mathbb{C}^n}/S$, the quotient bundle. They are all $T$-equivariant. 
\begin{definition}
	Let $\widetilde{\sigma}_r$ be $c_r^T(Q)$, the $r$-th equivariant Chern class of $Q$, and $\widetilde{\sigma}_{1^r}$ be $c_r^T(S^*)$, the $r$-th equivariant Chern class of $S^*$. 
\end{definition}
In fact $\widetilde{\sigma}_r$ is an equivariant lift of the ordinary Schubert class $\sigma_r$ corresponding to the row Young diagram with $r$ boxes, while $\widetilde{\sigma}_{1^r}$ is an equivariant lift of the ordinary Schubert class corresponding to the column Young diagram with $r$ boxes. It should be noted that $\widetilde{\sigma}_1$ and $\widetilde{\sigma}_{1^1}$ are not the same, though both lifts $\sigma_1$: taking the first equivariant Chern classes of vector bundles in the short exact sequence
\[0\longrightarrow S\longrightarrow\underline{\mathbb{C}^n}\longrightarrow Q\longrightarrow 0\]
gives
\begin{align}
	c_1^T(S)+c_1^T(Q)&=c_1^T(\underline{\mathbb{C}^n})\nonumber\\
	-c_1^T(S^*)+\sigma_1&=\sum_{i=1}^n t_i\nonumber\\
	\widetilde{\sigma}_{1^1}&=\widetilde{\sigma}_1-e_1.\label{diff}
\end{align}
In general, the Whitney product formula $\displaystyle c^T(S)c^T(Q)=c^T(\underline{\mathbb{C}^n})=1+\sum_{i=1}^ne_i$ gives the following
\begin{proposition}\label{colgiam}
	$\widetilde{\sigma}_{1^r}=\det(\widetilde{\sigma}_{1+j-i}-\delta_{rj}e_{1+j-i})_{1\leq i, j\leq r}$ (here we take the convention that $\widetilde{\sigma}_r=0$ if $r>n-k$ or $r<0$, and $\widetilde{\sigma}_0=1$).
\end{proposition}
\begin{proof}
	We shall prove by induction on $r$, and assume that the formula is true for $k\leq r-1$. The base case $k=1$ is already proved above. By the Whitney product formula, we have
	\[(1-(\widetilde{\sigma}_1-e_1)+c_2^T(S^*)+\cdots+(-1)^ic_i^T(S^*)+\cdots+(-1)^kc_k^T(S^*))(1+\widetilde{\sigma}_1+\cdots+\widetilde{\sigma}_{n-k})=1+e_1+\cdots+e_n\]
	\[(1-(\widetilde{\sigma}_1-e_1)+\widetilde{\sigma}_{1^2}+\cdots+(-1)^i\widetilde{\sigma}_{1^i}+\cdots+(-1)^k\widetilde{\sigma}_{1^k})(1+\widetilde{\sigma}_1+\cdots+\widetilde{\sigma}_{n-k})=1+e_1+\cdots+e_n\]
	Comparing the degree $2r$ terms of both sides, we have
	\begin{align*}
		\widetilde{\sigma}_r-(\widetilde{\sigma}_1-e_1)\widetilde{\sigma}_{r-1}+\widetilde{\sigma}_{1^2}\widetilde{\sigma}_{r-2}+\cdots+(-1)^i\widetilde{\sigma}_{1^i}\widetilde{\sigma}_{r-i}+\cdots+(-1)^r\widetilde{\sigma}_{1^r}&=e_r\\
		(-1)^{r+1}(\widetilde{\sigma}_r-e_r)+(-1)^{r+2}(\widetilde{\sigma}_1-e_1)\widetilde{\sigma}_{r-1}+\cdots+(-1)^{r+1+i}\widetilde{\sigma}_{1^i}\widetilde{\sigma}_{r-i}+\cdots+(-1) ^{2r}\widetilde{\sigma}_{1^{r-1}}\widetilde{\sigma}_1&=\widetilde{\sigma}_{1^r}
	\end{align*}
	By the inductive hypothesis, the left-hand side of the last equation is the cofactor expansion of the matrix $(\widetilde{\sigma}_{1+j-i}-\delta_{rj}e_{1+j-i})_{1\leq i, j\leq r}$ along the first row. This completes the inductive step.
\end{proof}
\begin{remark}
	Proposition \ref{colgiam} can be viewed as an equivariant extension of the Giambelli formula for $\sigma_{1^r}$ in terms of the special Schubert classes. 
\end{remark}
Using the equivariant Schubert classes $\widetilde{\sigma}_{1^r}$, one can define equivariant lifts of general Schubert classes as characteristic classes of vector bundles constructed out of $S^*$. Let $\lambda$ be a Young diagram and $\ell(\lambda)$ the number of rows of $\lambda$. 
\begin{definition}\label{equivGiambelli}
	\begin{enumerate}
		\item Let $\widehat{\sigma}_r=\det(\widetilde{\sigma}_{1^{1+j-i}})_{1\leq i, j\leq r}$ (here we take the convention that $\widetilde{\sigma}_{1^r}=0$ if $r>k$ or $r<0$, and $\widetilde{\sigma}_{1^r}=1$ if $r=0$).
		\item For a general Young diagram $\lambda$ (not necessarily contained in the $k\times(n-k)$ rectangle), let $\widehat{\sigma}_\lambda$ be $\det(\widehat{\sigma}_{\lambda_i+j-i})_{1\leq i, j\leq \ell(\lambda)}$, where $\ell(\lambda)$ is the number of rows of $\lambda$.  
	\end{enumerate}
\end{definition}
The definition of $\widehat{\sigma}_\lambda$ is just Giambelli's formula in terms of the special equivariant Schubert classes $\widehat{\sigma}_r$. 

\begin{remark}\label{coincide}
	Since $\widehat{\sigma}_{1^r}=\det(\widehat{\sigma}_{1+j-i})_{1\leq i, j\leq r}$, we may regard $\widehat{\sigma}_r$ as the $r$-th completely symmetric polynomial and $\widehat{\sigma}_{1^r}$ the $r$-th elementary symmetric polynomial. By $\widehat{\sigma}_\lambda=\det(\widehat{\sigma}_{\lambda_i+j-i})_{1\leq i, j\leq \ell{\lambda}}$ and the two Giambelli's formulae for the Schur polynomial $s_\lambda$ (cf. \cite[Appendix A, Equations (A5) and (A6)]{FH}), we have that $\widehat{\sigma}_\lambda=\det\left(\widehat{\sigma}_{1^{\lambda_i^T+j-i}}\right)_{1\leq i, j\leq\ell(\lambda^T)}$. We also have $\widehat{\sigma}_{1^r}=\widetilde{\sigma}_{1^r}$. In particular, $\widehat{\sigma}_{1^1}=\widehat{\sigma}_1$
\end{remark}
\begin{corollary}\label{vanishing}
	$\widehat{\sigma}_\lambda=\det(\widehat{\sigma}_{\lambda_i+j-i})_{1\leq i, j\leq\ell(\lambda)}=0$ if $\ell(\lambda)>k$. 
\end{corollary}
\begin{proof}
	By Remark \ref{coincide}, $\widehat{\sigma}_\lambda=\det\left(\widetilde{\sigma}_{1^{\lambda_i^T+j-i}}\right)_{1\leq i, j\leq\ell(\lambda^T)}$. Note that $\lambda_1^T=\ell(\lambda)>k$. Thus all entries $\widetilde{\sigma}_{1^{\lambda_1^T+j-1}}$ in the first row of the matrix are 0, and the result follows.
\end{proof}
\begin{remark}
	Let $\mathcal{S}_\lambda$ be the Schur functor corresponding to the Young diagram $\lambda$ which is applied to a vector bundle (cf. \cite[\S6.1]{FH}). By the (equivariant) splitting principle, we may assume that $S^*$ is the direct sum of equivariant line bundles $\displaystyle \bigoplus_{i=1}^k L_i$. Then 
\[\mathcal{S}_\lambda(S^*)=\bigoplus_{T\in\text{SSYT}(\lambda, k)}\bigotimes_{i\in T} L_i\]
where $\text{SSYT}(\lambda, k)$ is the set of semi-standard Young tableaux in the shape of $\lambda$ with labels $1, 2, \cdots, k$. For example, $\mathcal{S}_{1^r}(S^*)=\bigwedge\nolimits^r S^*$ and $\mathcal{S}_r(S^*)=\text{Sym}^r S^*$. Let $\displaystyle c_{\mathcal{S}_\lambda}^T(S^*):=\sum_{T\in\text{SSYT}(\lambda, k)}\prod_{i\in T}c_1^T(L_i)$, which is $s_\lambda(c_1^T(L_1), \cdots, c_1^T(L_k))$, where $s_\lambda$ is the Schur polynomial associated to $\lambda$. If $\lambda$ is a Young diagram inside the $k\times(n-k)$ rectangle, then it can be shown that $c_{\mathcal{S}_\lambda}^T(S^*)=\widehat{\sigma}_\lambda$. In fact Corollary \ref{vanishing} also follows from this observation because $\mathcal{S}_\lambda(S^*)$ is the zero vector bundle for $\ell(\lambda)>k$.  
\end{remark}
\begin{proposition}\label{ringstr}
	\begin{enumerate}
		\item\label{simplify} $\displaystyle\widetilde{\sigma}_r=\sum_{i=0}^r e_i\widehat{\sigma}_{r-i}$. Here $\widetilde{\sigma}_r=0$ if $r>n-k$ and $e_r=0$ if $r>n$. 
		\item \begin{align*}
				&H_T^*(X, \mathbb{Z})\\
				\cong&\frac{\mathbb{Z}[\widetilde{\sigma}_1, \widetilde{\sigma}_2, \cdots, \widetilde{\sigma}_{n-k}, \widetilde{\sigma}_{1^2}, \cdots, \widetilde{\sigma}_{1^k}, t_1, \cdots, t_n]}{\left((1-(\widetilde{\sigma}_1-e_1)+\widetilde{\sigma}_{1^2}+\cdots+(-1)^i\widetilde{\sigma}_{1^i}+\cdots+(-1)^k\widetilde{\sigma}_{1^k})(1+\widetilde{\sigma}_1+\widetilde{\sigma}_2+\cdots+\widetilde{\sigma}_{n-k})-1-\sum_{i=1}^n e_i\right)}\\
				\cong&\frac{\mathbb{Z}[\widetilde{\sigma}_1, \widetilde{\sigma}_2, \cdots, \widetilde{\sigma}_{n-k}, t_1, \cdots, t_n]}{(\widetilde{Y}_{k+1}, \widetilde{Y}_{k+2}, \cdots, \widetilde{Y}_{n})}\ (\text{where }\widetilde{Y}_r=\det(\widetilde{\sigma}_{1+j-i}-\delta_{rj}e_{1+j-i})_{1\leq i, j\leq r})\\
				\cong&\frac{\mathbb{Z}[\widehat{\sigma}_1, \widehat{\sigma}_2, \cdots, \widehat{\sigma}_{n-k}, t_1, \cdots, t_n]}{(\widehat{Y}_{k+1}, \widehat{Y}_{k+2}, \cdots, \widehat{Y}_{n})}\ (\text{where }\widehat{Y}_r=\det\left(\widehat{\sigma}_{1+j-i}\right)_{1\leq i, j\leq r}=\widetilde{Y}_r)
			\end{align*}
	\end{enumerate}
\end{proposition}
\begin{proof}
	\begin{enumerate}
		\item We prove by induction. The base case $\widetilde{\sigma}_1=\widehat{\sigma}_1+e_1$ holds by Remark \ref{coincide} and Equation (\ref{diff}). Assume that the equation $\widetilde{\sigma}_j=\sum_{i=0}^je_i\widehat{\sigma}_{j-i}$ holds for $1\leq j\leq r-1$. Using Definition \ref{equivGiambelli}, we write out Giambelli's formula for $\widehat{\sigma}_{1^r}$,
		\[\widehat{\sigma}_{1^r}=\left|\begin{matrix}\widehat{\sigma}_1&\widehat{\sigma}_2&\cdots&\widehat{\sigma}_{r-1}&\widehat{\sigma}_r\\ 1&\widehat{\sigma}_1&\cdots&\widehat{\sigma}_{r-2}&\widehat{\sigma}_{r-1}\\ \vdots&\vdots&\ddots&\vdots&\vdots\\ 0&0&\cdots&\widehat{\sigma}_1&\widehat{\sigma}_2\\ 0&0&\cdots&1&\widehat{\sigma}_1\end{matrix}\right|,\]
		and apply the following row operations in the indicated order.
		\begin{enumerate}
			\item To the $(r-1)$-st row, add $e_1\cdot r\text{-th row}$. 
			\item To the $(r-2)$-nd row, add $e_2\cdot r\text{-th row}+ e_1\cdot(r-1)\text{-st row}$.
			\item In general, to the $(r-j)$-th row, add $\displaystyle\sum_{i=0}^{j-1}e_{j-i}\cdot (r-i)\text{-th row}$.
		\end{enumerate}
		After applying the above row operations for $j$ ranging from 1 to $r-1$, and using the inductive hypothesis, we end up with the determinant
		\[\left|\begin{matrix}\widetilde{\sigma}_1&\widetilde{\sigma}_2&\cdots&\widetilde{\sigma}_{r-1}&\sum_{j=1}^r e_{r-j}\widetilde{\sigma}_j\\ 1&\widetilde{\sigma}_1&\cdots&\widetilde{\sigma}_{r-2}&\widetilde{\sigma}_{r-1}-e_{r-1}\\ \vdots&\vdots&\ddots&\vdots&\vdots\\ 0&0&\cdots&\widetilde{\sigma}_1&\widetilde{\sigma}_2-e_2\\ 0&0&\cdots& 1&\widetilde{\sigma}_1-e_1\end{matrix}\right|.\]
		By Remark \ref{coincide} and Proposition \ref{colgiam}, the class $\widetilde{\sigma}_{1^r}$ is also the determinant
		\[\left|\begin{matrix}\widetilde{\sigma}_1&\widetilde{\sigma}_2&\cdots&\widetilde{\sigma}_{r-1}&\widetilde{\sigma}_r-e_r\\ 1&\widetilde{\sigma}_1&\cdots&\widetilde{\sigma}_{r-2}&\widetilde{\sigma}_{r-1}-e_{r-1}\\ \vdots&\vdots&\ddots&\vdots&\vdots\\ 0&0&\cdots&\widetilde{\sigma}_1&\widetilde{\sigma}_2-e_2\\ 0&0&\cdots& 1&\widetilde{\sigma}_1-e_1\end{matrix}\right|.\]
		Taking the difference between these two determinants gives
		\[\left|\begin{matrix}0&0&\cdots&0&\widetilde{\sigma}_r-e_r-\sum_{j=1}^r e_{r-j}\widetilde{\sigma}_j\\ 1&\widetilde{\sigma}_1&\cdots&\widetilde{\sigma}_{r-2}&\widetilde{\sigma}_{r-1}-e_{r-1}\\ \vdots&\vdots&\ddots&\vdots&\vdots\\ 0&0&\cdots&\widetilde{\sigma}_1&\widetilde{\sigma}_2-e_2\\ 0&0&\cdots& 1&\widetilde{\sigma}_1-e_1\end{matrix}\right|=0.\]
		Applying the cofactor expansion along the first row immediately yields $\widetilde{\sigma}_r-e_r-\sum_{j=1}^r e_{r-j}\widetilde{\sigma}_j=0$, which completes the inductive step.
		\item It is known that the ordinary cohomology ring of $X$ is generated by the special Schubert classes and classes corresponding to a single column. In particular the ring is isomorphic to 
		\[\frac{\mathbb{Z}[\sigma_1, \sigma_2, \cdots, \sigma_{n-k}, \sigma_{1^2}, \cdots, \sigma_{1^k}]}{((1-\sigma_1+\cdots+(-1)^i\sigma_{1^i}+\cdots+(-1)^k\sigma_{1^k})(1+\sigma_1+\cdots+\sigma_{n-k})-1)}\]
		(cf. \cite[Proposition 23.2]{BT}). By the equivariant formality of the $T$-action on $X$, the equivariant cohomology ring $H_T^*(X, \mathbb{Z})$ is a $H_T^*(\text{pt}, \mathbb{Z})$-algebra generated by the equivariant lifts $\widetilde{\sigma}_1, \widetilde{\sigma}_2, \cdots, \widetilde{\sigma}_{n-k}, \widetilde{\sigma}_{1^2}, \cdots, \widetilde{\sigma}_{1^k}$, with the relation $\displaystyle c^T(S)c^T(Q)=c^T(\underline{\mathbb{C}^n})=1+\sum_{i=1}^n e_i$ which lifts the generator of the relation ideal of the above ring. That shows the first isomorphism. The second isomorphism follows by noting that $\widetilde{\sigma}_{1^2}, \cdots, \widetilde{\sigma}_{1^k}$ and the generator of the relation ideal can be expressed in terms of the special Schubert classes $\widetilde{\sigma}_i$, $1\leq i\leq n-k$ by Proposition \ref{colgiam}. One can instead use $\{\widehat{\sigma}_i\}_{i=1}^{n-k}$ as generators since by (\ref{simplify}) of this Proposition, $\{\widetilde{\sigma}_i\}_{i=1}^{n-k}$ and $\{\widehat{\sigma}_i\}_{i=1}^{n-k}$ can be expressed in terms of each other. Rewriting the generators of the relation ideal, which are basically $\widetilde{\sigma}_{1^i}(=\widehat{\sigma}_{1^i})$, $k+1\leq i\leq n$, in terms of $\{\widehat{\sigma}_i\}_{i=1}^{n-k}$, we get the last isomorphism.
	\end{enumerate}
\end{proof}
Because of the way the equivariant Schubert classes $\widehat{\sigma}_\lambda$ is defined, we can compute the product $\widehat{\sigma}_\lambda\cdot\widehat{\sigma}_\mu$ as if they were ordinary Schubert classes of $\text{Gr}(k, \infty)$ (we may just restrict our attention to those terms in the product with Young diagrams occupying less than or equal to $k$ rows because of Corollary \ref{vanishing}) using the Littlewood-Richardson rule, Pieri's formula and Giambelli's formula, and use  Proposition \ref{ringstr}(1) to reduce the product to a linear combination of those classes $\widehat{\sigma}_\nu$ where $\nu$ is contained in the $k\times(n-k)$ rectangle.
\begin{example}\label{equivexam}
	Let $X=\text{Gr}(2, 5)$. To compute $\widehat{\sigma}_{2, 1}\cdot\widehat{\sigma}_2$, we first apply Pieri's rule to the product as if the equivariant Schubert classes were ordinary ones for $\text{Gr}(2, \infty)$ and get $\widehat{\sigma}_{4, 1}+\widehat{\sigma}_{3, 2}$. We rewrite $\widehat{\sigma}_{4, 1}$ as $\widehat{\sigma}_4\cdot\widehat{\sigma}_1-\widehat{\sigma}_5$ by Giambelli's formula. By Proposition \ref{ringstr}(1) and Pieri's formula, we can further rewrite $\widehat{\sigma}_4\cdot\widehat{\sigma}_1-\widehat{\sigma}_5$ as 
	\begin{align*}
		&(-e_4-e_3\widehat{\sigma}_1-e_2\widehat{\sigma}_2-e_1\widehat{\sigma}_3)\cdot\widehat{\sigma}_1+e_5+e_4\widehat{\sigma}_1+e_3\widehat{\sigma}_2+e_2\widehat{\sigma}_3+e_1\widehat{\sigma}_4\\
		=&-e_4\widehat{\sigma}_1-e_3(\widehat{\sigma}_{1^2}+\widehat{\sigma}_2)-e_2(\widehat{\sigma}_3+\widehat{\sigma}_{2, 1})-e_1(\widehat{\sigma}_4+\widehat{\sigma}_{3, 1})+e_5+e_4\widehat{\sigma}_1+e_3\widehat{\sigma}_2+e_2\widehat{\sigma}_3+e_1\widehat{\sigma}_4\\
		=&e_5-e_3\widehat{\sigma}_{1^2}-e_2\widehat{\sigma}_{2, 1}-e_1\widehat{\sigma}_{3, 1}.
	\end{align*}
	So we have
	\[\widehat{\sigma}_{2, 1}\cdot\widehat{\sigma}_2=\widehat{\sigma}_{3, 2}-e_1\widehat{\sigma}_{3, 1}-e_2\widehat{\sigma}_{2, 1}-e_3\widehat{\sigma}_{1^2}+e_5.\]
\end{example}
\begin{theorem}[Equivariant Pieri's formula]\label{equivpieri}
	Let $\lambda$ be a Young diagram with $\ell$ rows in the $k\times(n-k)$ rectangle. Then 
	\begin{align*}
		\widehat{\sigma}_\lambda\cdot\widehat{\sigma}_1&=\sum_{\lambda'\in\lambda\oplus 1}\widehat{\sigma}_{\lambda'}+\delta_{\lambda_1, n-k}\left(\sum_{j=0}^{\ell-2}(-1)^{\ell-j}\sum_{m=\lambda_{\ell-j+1}}^{\lambda_{\ell-j}-1}e_{n-k+\ell-m-j}\widehat{\sigma}_{\lambda_2-1, \lambda_3-1, \cdots, \lambda_{\ell-j}-1, m, \lambda_{\ell-j+1}, \cdots, \lambda_\ell}\right.\\
		&\left.-\sum_{m=\lambda_2}^{n-k}e_{n-k-m+1}\cdot\widehat{\sigma}_{m, \lambda_2, \cdots, \lambda_\ell}\right)
	\end{align*}
	where $\lambda\oplus 1$ is the set of Young diagrams in the $k\times (n-k)$ rectangle obtained by adding to $\lambda$ one box.
\end{theorem}
\begin{proof}
	By the classical Pieri's formula, $\displaystyle\widehat{\sigma}_\lambda\cdot\widehat{\sigma}_1=\sum_{\lambda'\in\lambda\oplus 1}\widehat{\sigma}_{\lambda'}$ if $\lambda'<n-k$, and $\displaystyle\widehat{\sigma}_\lambda\cdot\widehat{\sigma}_1=\sum_{\lambda'\in\lambda\oplus 1}\widehat{\sigma}_{\lambda'}+\widehat{\sigma}_{n-k+1, \lambda_2, \cdots, \lambda_\ell}$ if $\lambda_1=n-k+1$. By Giambelli's formula, and Proposition \ref{ringstr}(\ref{simplify}), 
	\begin{align*}
		\widehat{\sigma}_{n-k+1, \lambda_2, \cdots, \lambda_\ell}&=\left|\begin{matrix}\widehat{\sigma}_{n-k+1}&\widehat{\sigma}_{n-k+2}&\cdots&\widehat{\sigma}_{n-k+\ell}\\ \widehat{\sigma}_{\lambda_2-1}&\widehat{\sigma}_{\lambda_2}&\cdots&\widehat{\sigma}_{\lambda_2+\ell-2}\\ \vdots&\vdots&\ddots&\vdots\\ \widehat{\sigma}_{\lambda_\ell-\ell+1}&\widehat{\sigma}_{\lambda_\ell-\ell+2}&\cdots&\widehat{\sigma}_{\lambda_\ell}\end{matrix}\right|\\
		&=\left|\begin{matrix}-\sum_{j=0}^{n-k}e_{n-k+1-j}\widehat{\sigma}_j&-\sum_{j=0}^{n-k+1}e_{n-k+2-j}\widehat{\sigma}_j&\cdots&-\sum_{j=0}^{n-k+\ell+1}e_{n-k+\ell-j}\widehat{\sigma}_j\\ \widehat{\sigma}_{\lambda_2-1}&\widehat{\sigma}_{\lambda_2}&\cdots&\widehat{\sigma}_{\lambda_2+\ell-2}\\ \vdots&\vdots&\ddots&\vdots\\ \widehat{\sigma}_{\lambda_\ell-\ell+1}&\widehat{\sigma}_{\lambda_\ell-\ell+2}&\cdots&\widehat{\sigma}_{\lambda_\ell}\end{matrix}\right|\\
		&=-\sum_{p=0}^{n-1}e_{n-k+\ell-p}\left|\begin{matrix}\widehat{\sigma}_{p-\ell+1}&\widehat{\sigma}_{p-\ell+2}&\cdots&\widehat{\sigma}_p\\ \widehat{\sigma}_{\lambda_2-1}&\widehat{\sigma}_{\lambda_2}&\cdots&\widehat{\sigma}_{\lambda_2+\ell-2}\\ \vdots&\vdots&\ddots&\vdots\\ \widehat{\sigma}_{\lambda_\ell-\ell+1}&\widehat{\sigma}_{\lambda_\ell-\ell+2}&\cdots&\widehat{\sigma}_{\lambda_\ell}\end{matrix}\right|
	\end{align*}
	Rewriting the last sum as a sum over $j$ where $\ell-j$ is the row to which the first row vector $(\widehat{\sigma}_{p-\ell+1}, \widehat{\sigma}_{p-\ell+2}\cdots, \widehat{\sigma}_p)$ (here $\widehat{\sigma}_i=0$ if $i<0$) is moved by row operation so that the subscripts of the diagonal entries of the new matrix is in descending order down the row, and applying Giambelli's formula, we get the second term of the RHS of the equivariant Pieri's formula. 
\end{proof}
As is the case for classical Schubert calculus, there is also a duality isomorphism between the equivariant cohomology rings of $X$ and $X':=\text{Gr}(n-k, n)$. Let 
\begin{align*}
	\text{inv}: X'&\to X\\
	[V]&\mapsto[V^\perp],
\end{align*}
where the perp is taken with respect to the standard Hermitian inner product of $\mathbb{C}^n$.
\begin{definition}\label{dualchar}
	Define $\sigma_{1^r}':=\det(\widetilde{\sigma}_{1+j-i})_{1\leq i, j\leq r}$, and for a general Young diagram $\lambda$ (not necessarily contained in the $k\times(n-k)$ rectangle), 
	\[\sigma'_\lambda:=\det\left(\sigma'_{1^{\lambda_i^T+j-i}}\right)_{1\leq i, j\leq\ell(\lambda^T)}, \]
	where $\lambda^T$ is the conjugate Young diagram of $\lambda$. 
\end{definition}
\begin{proposition}\label{dual}
	The map $\text{inv}$ induces an isomorphism of equivariant cohomology rings
	\[\text{inv}^*: H_T^*(X, \mathbb{Z})\to H_T^*(X', \mathbb{Z})\]
	and $\text{inv}^*(\widehat{\sigma}_\lambda)=\sigma'_{\lambda^T}$, $\text{inv}^*(t_i)=t_i$ for $1\leq i\leq n$. 
\end{proposition}
\begin{proof}
	The map $\text{inv}$ is a $T$-equivariant diffeomorphism and so the induced map is an isomorphism of equivariant cohomology rings. Note that $\text{inv}^*S^*=Q'$, where $Q'$ is the universal quotient bundle of $X'$. It follows that
	\begin{align*}
		\text{inv}^*\widetilde{\sigma}_{1^r}&=\text{inv} ^* c_r^T(S^*)\\
									&=c_r ^T(Q')\\
									&=\widetilde{\sigma}_r\\
									&=\sigma'_r\\
		\text{inv} ^*\widehat{\sigma}_r&=\text{inv} ^*\det(\widehat{\sigma}_{1^{1+j-i}})_{1\leq i, j\leq r}\\
								&=\det(\sigma'_{1+j-i})_{1\leq i, j\leq r}\\
								&=\sigma'_{1^r}\\
		\text{inv} ^*\widehat{\sigma}_\lambda&=\text{inv} ^*\det(\widehat{\sigma}_{\lambda_i+j-i})_{1\leq i, j\leq \ell(\lambda)}\\
									&=\det\left(\sigma'_{1^{\lambda_i+j-i}}\right)_{1\leq i, j\leq \ell(\lambda)}\\
									&=\sigma'_{\lambda^T}
	\end{align*}
\end{proof}
\begin{remark}\label{dualprop}
Proposition \ref{dual} allows one to get the dual version of previous results.
\begin{enumerate}
	\item Applying $\text{inv}^*$ to the algebra presentation of $H_T^*(X, \mathbb{Z})$ in Proposition \ref{ringstr} gives
	\[H_T^*(X', \mathbb{Z})\cong\frac{\mathbb{Z}[\sigma'_1, \sigma'_{1^2}, \cdots, \sigma'_{1^{n-k}}, t_1, \cdots, t_n]}{(Y'_{k+1}, \cdots, Y'_{n})}, \]
	where $Y'_r=\det\left(\sigma'_{1^{1+j-i}}\right)_{1\leq i, j\leq r}$. This suggests that we can use the column equivariant Schubert classes $\sigma'_{1^r}$ as generators to present $H_T^*(X, \mathbb{Z})$ as 
	\[\frac{\mathbb{Z}[\sigma'_1, \sigma'_{1^2}, \cdots, \sigma'_{1^k}, t_1, \cdots, t_n]}{(Y_{n-k+1}', \cdots, Y'_{n})}.\]
	\item\label{colvanishing} $\sigma'_\lambda=c_{\mathcal{S}_{\lambda^T}}^T(Q)$. If $\sigma'_\lambda\in H_T^*(X, \mathbb{Z})$ and $\ell(\lambda^T)>n-k$, then $\sigma'_\lambda=0$. 
	\item\label{colrecursive} We have $\widetilde{\sigma}_{1^r}=\sum_{i=0}^r e_i\sigma'_{1^{r-i}}$. 
	\item To compute $\sigma'_\lambda\cdot\sigma'_\mu$ in $H_T^*(X, \mathbb{Z})$, we can multiply the equivariant Schubert classes as if they were ordinary ones and use (\ref{colvanishing}), (\ref{colrecursive}) and equivariant Giambelli's formula to simplify the product to a linear combination of $\sigma'_\nu$ where $\nu$ is contained in the $k\times(n-k)$ rectangle.
\end{enumerate}
\end{remark}
\section{Quantum Schubert calculus and derivation of quantum multiplication rules}\label{qsch}
\subsection{Algebra presentation of quantum cohomology ring}
The (small) quantum cohomology ring $QH^*(X, \mathbb{Z})$ is the cohomology ring $H^*(X, \mathbb{Z})$ deformed by the parameter $q$ which records the `quantum' intersection number of Schubert varieties (aka Gromov-Witten invariants). More precisely, $QH^*(X, \mathbb{Z})$ is a $\mathbb{Z}[q]$-algebra isomorphic to $\mathbb{Z}[q]\otimes_\mathbb{Z}H^*(X, \mathbb{Z})$ as a module over $\mathbb{Z}[q]$, with $q$ of degree $n$ and $\sigma_\lambda$ of degree $|\lambda|$, and its ring structure is defined by the quantum product
\[\sigma_\lambda*\sigma_\mu=\sum_{\nu,\ d\geq 0}\langle\sigma_\lambda, \sigma_\mu, \sigma_{\nu^\vee}\rangle_d\sigma_\nu q^d.\]
Here $\nu^\vee$ is the Young diagram complement to $\nu$ in the $k\times(n-k)$ rectangle, and $\langle\sigma_\lambda, \sigma_\mu, \sigma_{\nu^\vee}\rangle_d$ is the Gromov-Witten invariant counting the number of degree $d$ rational maps $p$ from $\mathbb{P}^1$ to $X$ such that $p(\mathbb{P}^1)$ intersects transversally Schubert varieties $X_\lambda$, $X_\mu$ and $X_{\nu^\vee}$ at $p(0)$, $p(1)$ and $p(\infty)$ respectively. The quantum cohomology ring admits a compact presentation which is similar to that for equivariant cohomology in Proposition \ref{ringstr}(2).
\begin{theorem}[\cite{W, ST}]\label{ST}
	$\displaystyle QH^*(X, \mathbb{Z})\cong\frac{\mathbb{Z}[\sigma_1, \sigma_2, \cdots, \sigma_{n-k}, q]}{(Y_{k+1}, Y_{k+2}, \cdots, Y_{n-1}, Y_n+(-1)^{n-k}q)}$ (here $Y_r=\det(\sigma_{1+j-i})_{1\leq i, j\leq r}$, and we take the convention that $\sigma_r=0$ for $r>n-k$).
\end{theorem}
Siebert and Tian showed in \cite{ST} that the quantum cohomology ring of a Fano manifold (of which a Grassmannian is an example) can be obtained in a canonical way from the algebra presentation of its ordinary cohomology ring: if $X_1, \cdots, X_N$ are generators and $f_1, \cdots, f_k$ relations of the ordinary cohomology ring, then the quantum cohomology ring can be described using the same set of generators, and the relations are got by evaluating $f_1, \cdots, f_k$  in the quantum cohomology ring. With this result, the derivation of the algebra presentation in Theorem \ref{ST} is reduced to the computation of the single quantum product
\[\sigma_{n-k}*\sigma_{1^k},\]
which was obtained earlier by Witten (cf. \cite{W}) and is associated to the evaluation of $Y_n$ in the quantum cohomology ring and attributed to the last equation in the relation ideal. 

We will recover some previously known quantum multiplication rules (quantum Pieri's formula \cite{Ber, BCFF}, and the rim hook algorithm \cite[Main Lemma]{BCFF}) in slightly different forms. The inputs of the proof are results of equivariant Schubert calculus in the previous section, Theorem \ref{ST} and quantum Giambelli's formula (cf. \cite{Ber}), which says that any general Schubert class $\sigma_\lambda\in QH^*(X, \mathbb{Z})$ for $\lambda$ contained in the $k\times(n-k)$ rectangle can be expressed in terms of special Schubert classes by the ordinary Giambelli's formula, without any quantum deformation. 
\begin{remark}
	In \cite{Ber}, Bertram obtained quantum Pieri's and Giambelli's formulae by analysing Grothendieck's quot schemes, without appealing to Theorem \ref{ST} and the theorem of associativity of quantum multiplication in general (cf. \cite{RT}), the latter of which is highly nontrivial and foundational in quantum cohomology theory. In fact, both the quantum cohomology ring structure and the associativity theorem for Grassmannians can be deduced from quantum Pieri and Giambelli (cf. \cite[Final remarks]{Ber}). We shall point out that our proof implicitly makes use of the associativity theorem, which is already assumed by Theorem \ref{ST}.
\end{remark}
\subsection{Linking equivariant and quantum cohomology}\label{linking}
In this section, we will work with instead the different but related equivariant cohomology ring $H_{U(n)}^*(X, \mathbb{Z})$, where the $U(n)$-action on $X$ is induced by its standard action on $\mathbb{C}^n$. As the $U(n)$-action restricts to the action by the maximal torus $T$, $H_{U(n)}^*(X, \mathbb{Z})$ can be realized as the injective image of the restriction map $H_{U(n)}^*(X, \mathbb{Z})\to H_T^*(X, \mathbb{Z})$, which more precisely is the subring of $H_T^*(X, \mathbb{Z})$ invariant under the Weyl group action. So 
\[H_{U(n)}^*(X, \mathbb{Z})\cong\frac{\mathbb{Z}[\widehat{\sigma}_1, \widehat{\sigma}_2, \cdots, \widehat{\sigma}_{n-k}, e_1, \cdots, e_n]}{(\widehat{Y}_{k+1}, \widehat{Y}_{k+2}, \cdots, \widehat{Y}_{n})}.\]
The following is the key lemma which links the equivariant and quantum cohomology of $X$.
\begin{lemma}\label{homomorphism}
	The map 
	\begin{align*}
		f: H^*_{U(n)}(X, \mathbb{Z})&\to QH^*(X, \mathbb{Z})
	\end{align*}
	which sends 
	\begin{enumerate}
		\item $e_1, \cdots, e_{n-1}$ to 0, 
		\item $e_n$ to $(-1)^k q$, and
		\item $\widehat{\sigma}_i$ to $\sigma_i$, $1\leq i\leq n-k$,
	\end{enumerate}
	defines a ring homomorphism. 
\end{lemma}
\begin{proof}
Now that $f$ is defined by the images of the algebra generators of $H_{U(n)}^*(X, \mathbb{Z})$, it suffices to show that $f$ takes each of $\widehat{Y}_r$, $k+1\leq \ell\leq n$ to the relation ideal of $QH^*(X, \mathbb{Z})$. Note that for $n-k<r<n$, 
\begin{align*}
	f(\widehat{\sigma}_r)&=f\left(\widetilde{\sigma}_r-\sum_{i=1}^r e_i\widehat{\sigma}_{r-i}\right)\ (\text{by Proposition \ref{ringstr}(1)})\\
					&=0\ (\widetilde{\sigma}_r=0)\\
	f(\widehat{\sigma}_n)&=f\left(\widetilde{\sigma}_n-\sum_{i=1}^ne_i\widehat{\sigma}_{n-i}\right)\\
					&=f(-e_n)\\
					&=(-1)^{k+1}q.
\end{align*}
 Thus for $k+1\leq r<n$, 
 \begin{align*}
 	f(\widehat{Y}_r)&=Y_r,\ \text{and}\\
	f(\widehat{Y}_n)&=\left|\begin{matrix}\sigma_1&\sigma_2&\cdots&\sigma_{n-k}&0&\cdots&0&(-1)^{k+1}q\\ 1&\sigma_1&\cdots&\sigma_{n-k-1}&\sigma_{n-k}&\cdots&0&0\\ \vdots&\vdots&\ddots&\vdots&\vdots&\ddots&\vdots&\vdots\\ 0&0&\cdots&0&0&\cdots&1&\sigma_1\end{matrix}\right|\\
	&=\left|\begin{matrix}\sigma_1&\sigma_2&\cdots&\sigma_{n-k}&0&\cdots&0&0\\ 1&\sigma_1&\cdots&\sigma_{n-k-1}&\sigma_{n-k}&\cdots&0&0\\ \vdots&\vdots&\ddots&\vdots&\vdots&\ddots&\vdots&\vdots\\ 0&0&\cdots&0&0&\cdots&1&\sigma_1\end{matrix}\right|+\left|\begin{matrix}0&0&\cdots&0&0&\cdots&0&(-1)^{k+1}q\\ 1&\sigma_1&\cdots&\sigma_{n-k-1}&\sigma_{n-k}&\cdots&0&0\\ \vdots&\vdots&\ddots&\vdots&\vdots&\ddots&\vdots&\vdots\\ 0&0&\cdots&0&0&\cdots&1&\sigma_1\end{matrix}\right|\\
	&=Y_n+(-1) ^{n+1}\cdot(-1) ^{k+1}q\\
	&=Y_n+(-1) ^{n-k}q.
 \end{align*}
\end{proof}
\begin{remark}\label{rowrimhook}
	 In general, if $r=\ell n+p$, $0\leq p<n$, we can similarly get $f(\widehat{\sigma}_r)=(-1)^{\ell(k+1)}q^\ell\sigma_{p}$.
\end{remark}
	The ring homomorphism $f$ can be seen as a `partially forgetful map', forgetting part of the equivariant structure by ignoring the equivariant variables except the top universal Chern class $e_n$. Thus it is expected that $f$ sends a general equivariant Schubert class $\widehat{\sigma}_\lambda$ to a possibly nontrivial quantum deformation of the ordinary Schubert class $\sigma_\lambda$. It turns out that there is no quantum deformation due to quantum Giambelli's formula.
\begin{proposition}\label{corr}
	For $\lambda$ contained in the $k\times(n-k)$ rectangle, $f(\widehat{\sigma}_\lambda)=\sigma_\lambda$.
\end{proposition}
\begin{proof}
	Simply note that
	\begin{align*}
		f(\widehat{\sigma}_\lambda)&=f(\det(\widehat{\sigma}_{\lambda_i+j-i})_{1\leq i, j\leq\ell(\lambda)})\\
							&=\det(f(\widehat{\sigma}_{\lambda_i+j-i}))_{1\leq i, j\leq\ell(\lambda)}\\
							&=\det(\sigma_{\lambda_i+j-i})_{1\leq i, j\leq \ell(\lambda)}\\
							&=\sigma_\lambda\ (\text{by quantum Giambelli's formula}).
	\end{align*}
\end{proof}
\begin{definition}\label{genschu}
	For a general Young diagram $\lambda$ not necessarily contained in the $k\times(n-k)$ rectangle, define $\sigma_\lambda:=f(\widehat{\sigma}_\lambda)$, i.e., $\sigma_\lambda=\det(f(\widehat{\sigma}_{\lambda_i+j-i}))_{1\leq i, j\leq\ell(\lambda)}$.
\end{definition}
\begin{remark}\label{rowrimhookgen}
When $n-k<r<n$, then $\sigma_r=0$. More generally, by Remark \ref{rowrimhook} we have $\sigma_r=(-1)^{\ell(k+1)}q^\ell\sigma_p$. By Corollary \ref{vanishing}, ${\sigma}_\lambda=0$ if $\ell(\lambda)>k$. 
\end{remark}
\emph{\textbf{Caveat:}} It should be noted that the convention indicated in Remark \ref{rowrimhookgen} is different from the one used in Theorem \ref{ST}, i.e., $\sigma_r=0$ for $r>n-k$. The relation $Y_n+(-1)^{n-k}q$, when expressed using the convention in Remark \ref{rowrimhookgen}, becomes $\det(\sigma_{1+j-i})_{1\leq i, j\leq n}$. In the remainder of this paper we will adopt the convention in Remark \ref{rowrimhookgen}.

With Lemma \ref{homomorphism}, Proposition \ref{corr} and Remark \ref{rowrimhook}, one can compute quantum products through equivariant cohomology. 
\begin{corollary}
	\begin{enumerate}
		\item Let $\widehat{c}_{\lambda\mu}^\nu$ and $c_{\lambda\mu}^\nu$ be equivariant and quantum Littlewood-Richardson coefficients respectively. Then 
	\[f(\widehat{c}_{\lambda\mu}^\nu)=c_{\lambda\mu}^\nu.\]
		\item To compute $\sigma_\lambda*\sigma_\mu$, we can multiply the Schubert classes as if they were ordinary ones for $\text{Gr}(k, \infty)$, and then apply quantum Giambelli's formula and Remark \ref{rowrimhook} to simplify the product to a linear combination of $\sigma_\nu$ where $\nu$ is contained inside the $k\times(n-k)$ rectangle.
	\end{enumerate}
\end{corollary}
\begin{example}
	From Example \ref{equivexam}, we have, for $X=\text{Gr}(2, 5)$, the quantum product $\sigma_{2, 1}*\sigma_2=\sigma_{3, 2}+q$ after applying $f$ to the equivariant product. Equivalently, we can multiply $\sigma_{2, 1}$ and $\sigma_2$ as if they were Schubert classes of $\text{Gr}(2, \infty)$. Then we get $\sigma_{2, 1}*\sigma_2=\sigma_{4, 1}+\sigma_{3, 2}$. We apply quantum Giambelli's formula to rewrite the first term:
	\begin{align*}
		\sigma_{4, 1}&=\left|\begin{matrix}\sigma_4&\sigma_5\\ 1&\sigma_1\end{matrix}\right|\\
		&=\left|\begin{matrix}0&(-1)^{2+1}q\\1&\sigma_1 \end{matrix}\right|\\
		&=q.
	\end{align*}
	Now we arrive at the same result. 
\end{example}
The quantum contribution in $\sigma_\lambda*\sigma_\mu$ all comes from those Schubert classes corresponding to Young diagrams $\nu$ with $\ell(\nu)\leq k$ and $\nu_1>n-k$. This observation leads to 
\begin{proposition}\label{colnodef}
	If the sum of numbers of columns of $\lambda$ and $\mu$ does not exceed $n-k$, then there is no quantum deformation in the product $\sigma_\lambda*\sigma_\mu$. 
\end{proposition}
\begin{proof}
	By Littlewood-Richardson rule, the number of columns of the Young diagram of any Schubert class appearing in the classical product $\sigma_\lambda\cdot\sigma_\mu\in H^*(\text{Gr}(k, \infty), \mathbb{Z})$ does not exceed the sum of numbers of columns of $\lambda$ and $\mu$, which is assumed to not exceed $n-k$. Thus there is no quantum deformation in the product $\sigma_\lambda*\sigma_\mu$. 
\end{proof}
\subsection{Duality}
\begin{proposition}\label{dualforget}
	The map $f: H_T^*(X, \mathbb{Z})\to QH^*(X, \mathbb{Z})$ satisfies $f(\sigma'_\lambda)=\sigma_\lambda$ for $\lambda$ contained in the $k\times(n-k)$ rectangle. 
\end{proposition}
\begin{proof}
	By Proposition \ref{ringstr}(1) and the definition of $f$, we have, for $r\leq n-1$, 
	\[f(\widetilde{\sigma}_r)=f(\widehat{\sigma}_r)=\sigma_r.\]
	Then under the same condition for $r$, 
	\begin{align*}
		f(\sigma'_{1^r})&=\det(f(\widetilde{\sigma}_{1+j-i}))_{1\leq i, j\leq r}\\
					&=\det(\sigma_{1+j-i})_{1\leq i, j\leq r}\\
					&=\sigma_{1^r}.
	\end{align*}
	If $\lambda$ is contained in the $k\times(n-k)$ rectangle, $\lambda_i^T+j-i\leq n-1$ for $1\leq i, j\leq\ell(\lambda^T)$. It follows that 
	\begin{align*}
		f(\sigma'_\lambda)&=\det(f(\sigma'_{1^{\lambda_i^T+j-i}}))_{1\leq i, j\leq \ell(\lambda^T)}\\
						&=\det(\sigma_{1^{\lambda_i^T+j-i}})_{1\leq i, j\leq \ell(\lambda^T)}, 
	\end{align*}
	which is the same as the other Giambelli's formula $\det(\sigma_{\lambda_i+j-i})_{1\leq i, j\leq \ell(\lambda)}$ for $\sigma_\lambda$ (cf. \cite[Appendix A, Equations (A.5) and (A.6)]{FH}).
\end{proof}
\begin{remark}
	If the condition that $\lambda$ is contained in the $k\times(n-k)$ rectangle is removed, then $f(\sigma'_\lambda)$ may not be $\sigma_\lambda$. For example, $\sigma_{1^n}=0$ while $f(\sigma'_{1^n})=(-1)^{n-k-1}q$. 
\end{remark}
There is a duality isomorphism between the quantum cohomology ring $QH^*(X, \mathbb{Z})$ and $QH^*(X', \mathbb{Z})$ (cf. \cite[Proposition 4.1]{BCFF}), similar to the equivariant version (Proposition \ref{dual}). We give below an alternative proof of the duality isomorphism by means of the equivariant duality. 
\begin{proposition}\textnormal(\cite[Proposition 4.1]{BCFF}\textnormal).\label{qduality} The group isomorphism 
\[\text{qinv}: QH^*(X, \mathbb{Z})\to QH^*(X', \mathbb{Z})\]
determined by $\text{qinv}(\sigma_\lambda)=\sigma_{\lambda^T}$ for $\lambda$ contained in the $k\times(n-k)$ rectangle and $\text{qinv}(q)=q$ is a ring isomorphism.
\end{proposition}
\begin{proof}
	Applying $\text{inv}^*$ to $\displaystyle\widehat{\sigma}_\lambda\cdot\widehat{\sigma}_\mu=\sum_{\nu\subseteq k\times(n-k)}\widehat{c}_{\lambda\mu}^\nu\widehat{\sigma}_\nu$ gives $\displaystyle \sigma'_{\lambda^T}\cdot\sigma'_{\mu^T}=\sum_{\nu\subseteq k\times(n-k)}\widehat{c}_{\lambda\mu}^\nu\sigma'_{\nu^T}$ by Proposition \ref{dual}. Applying $f$ to these two equations yields 
	\begin{eqnarray}
		\sigma_\lambda*\sigma_\mu=\sum_{\nu\subseteq k\times(n-k)} c_{\lambda\mu}^\nu\sigma_\nu, \label{2}\\
		\sigma_{\lambda^T}*\sigma_{\mu^T}=\sum_{\nu\subseteq k\times(n-k)}c_{\lambda\mu}^\nu\sigma_{\nu^T}.\label{3}
	\end{eqnarray}
	by Proposition \ref{dualforget}. The map qinv takes Equation (\ref{2}) to (\ref{3}) and is indeed a ring isomorphism.
\end{proof}
Proposition \ref{qduality} enables us to dualize some previous results. For example, dualizing Proposition \ref{colnodef} recovers the following
\begin{proposition}\textnormal(\cite[Lemma 3]{Bu}\textnormal). If the sum $\ell(\lambda)+\ell(\mu)$ does not exceed $k$, then there is no quantum deformation in the product $\sigma_\lambda*\sigma_\mu$. \end{proposition}
\subsection{Quantum multiplication rules revisited}

\begin{proof}[Proof of Theorem \ref{qpieri}]
	\begin{enumerate}
		\item This follows from Proposition \ref{colnodef}. 
		\item By the classical Pieri's formula, 
		\[\sigma_\lambda*\sigma_{1^r}=\sum_{\substack{\lambda_i'=\lambda_i\text{ or }\lambda_i+1, 1\leq i\leq k\\ |\lambda'|=|\lambda|+r}}\sigma_{\lambda'}.\]
		The quantum contribution comes from the sum
		\[\sum_{\substack{\lambda_1'=n-k+1\\ \lambda_i'=\lambda_i\text{ or }\lambda_i+1, 2\leq i\leq k\\|\lambda'^-|=|\lambda^-|+r-1}}\sigma_{\lambda'}.\]
		By quantum Giambelli's formula, for those $\lambda'$ appearing in the above sum, 
		\begin{align*}
			\sigma_{\lambda'}&=\left|\begin{matrix}\sigma_{n-k+1}&\sigma_{n-k+2}&\cdots&\sigma_{n-1}&\sigma_{n}\\ \sigma_{\lambda'_2-1}&\sigma_{\lambda_2'}&\cdots& \sigma_{\lambda'_2+k-3}&\sigma_{\lambda'_2+k-2}\\ \vdots&\vdots&\ddots&\vdots&\vdots\\ \cdots&\cdots&\cdots&\sigma_{\lambda'_k-1}&\sigma_{\lambda'_k}\end{matrix}\right|\\
						&=\left|\begin{matrix}0&0&\cdots&0&(-1)^{k+1}q\\ \sigma_{\lambda'_2-1}&\sigma_{\lambda_2'}&\cdots& \sigma_{\lambda'_2+k-3}&\sigma_{\lambda'_2+k-2}\\ \vdots&\vdots&\ddots&\vdots&\vdots\\ \cdots&\cdots&\cdots&\sigma_{\lambda'_k-1}&\sigma_{\lambda'_k}\end{matrix}\right|\ (\text{by Remark \ref{rowrimhookgen}})\\
						&=q\sigma_{\lambda_2'-1, \cdots, \lambda'_k-1}\ (\text{by cofactor expansion along the first row})
		\end{align*}
		So the quantum contribution can be rewritten as
		\begin{align*}
			&q\sum_{\substack{\lambda'_i=\lambda_i\text{ or }\lambda_i+1, 2\leq i\leq k\\ |\lambda'^-|=|\lambda^-|+r-1}}\sigma_{\lambda_2'-1, \cdots, \lambda_k'-1}\\
			=&q\sum_{\substack{\mu_i=\lambda_{i+1}-1\text{ or }\lambda_{i+1}, 1\leq i\leq k-1\\ |\mu|=|\lambda^-|-(k-r)}}\sigma_\mu\\
			=&q\sum_{\mu\in\lambda^-\ominus 1^{k-r}}\sigma_\mu,
		\end{align*}
		which is what we desire. 
	\end{enumerate}
\end{proof}
\begin{example}
	Let $X=\text{Gr}(5, 10)$, $\lambda=(5, 5, 4, 3, 3)$, $r=2$. The quantum contribution of $\sigma_\lambda*\sigma_{1^2}$, according to Theorem \ref{qpieri}, is given by the Young diagrams below. \\
	\\
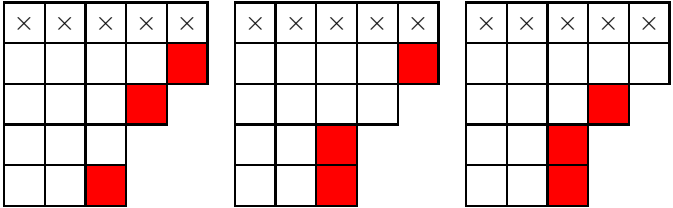
	\begin{ytableau}*(white)\times&*(white)\times&*(white)\times&*(white)\times&*(white)\times\\ *(white)&*(white)&*(white)&*(white)&*(red)\\ *(white)&*(white)&*(white)&*(red)\\*(white)&*(white)&*(white)\\*(white)&*(white)&*(red)\end{ytableau}\ \ \ \begin{ytableau}*(white)\times&*(white)\times&*(white)\times&*(white)\times&*(white)\times\\ *(white)&*(white)&*(white)&*(white)&*(red)\\ *(white)&*(white)&*(white)&*(white)\\*(white)&*(white)&*(red)\\*(white)&*(white)&*(red)\end{ytableau}\ \ \ 
	\begin{ytableau}*(white)\times&*(white)\times&*(white)\times&*(white)\times&*(white)\times\\ *(white)&*(white)&*(white)&*(white)&*(white)\\ *(white)&*(white)&*(white)&*(red)\\ *(white)&*(white)&*(red)\\ *(white)&*(white)&*(red)\end{ytableau}\\
	\\
	Together with the classical Pieri's formula, we have $\sigma_\lambda*\sigma_{1^2}=\sigma_{5, 5, 5, 4, 3}+\sigma_{5, 5, 4, 4, 4}+q(\sigma_{4, 3, 3, 2}+\sigma_{4, 4, 2, 2}+\sigma_{5, 3, 2, 2})$.
\end{example}
By duality (Proposition \ref{qduality}), we have the following quantum Pieri's rule for multiplying by a special Schubert class. 
\begin{corollary}\label{rqpieri}
	Consider the quantum product $\sigma_\lambda*\sigma_{r}$, where $\lambda$ is inside the $k\times(n-k)$ rectangle and $1\leq r\leq n-k$. 
	\begin{enumerate}
		\item If $\ell(\lambda)<k$, then $\sigma_\lambda*\sigma_{r}$ does not have any quantum deformation, i.e., 
		\[\sigma_\lambda*\sigma_{r}=\sum_{\lambda'\in\lambda\oplus r}\sigma_{\lambda'}, \]
		where $\lambda\oplus r$ is the set of Young diagrams in the $k\times(n-k)$ rectangle obtained by adding to $\lambda$ $r$ boxes, no two of which are in the same column.
		\item If $\ell(\lambda)=k$, then 
		\[\sigma_\lambda*\sigma_{r}=\sum_{\lambda'\in\lambda\oplus r}\sigma_{\lambda'}+q\sum_{\mu\in\lambda\ominus 1^k\ominus{(n-k-r)}}\sigma_\mu,\]
		where $\lambda\ominus 1^k\ominus{(n-k-r)}$ is the set of Young diagrams obtained by removing one box from each of any $n-k-r$ columns of the Young diagram $\lambda\ominus 1^k:=(\lambda_1-1, \lambda_2-1, \cdots, \lambda_{k}-1)$.
	\end{enumerate}
\end{corollary}
\begin{remark}
	The quantum part of the quantum Pieri's formula in \cite[Proposition 4.2]{BCFF} for $\sigma_\lambda*\sigma_{1^r}$ is described as a summation over the set of Young diagrams obtained by removing $n-r$ boxes in the southeast boundary of $\lambda$, at least one box from each column of $\lambda$. In fact this description is equivalent to that in Theorem \ref{qpieri} because the set of these Young diagrams is in fact the same as $\lambda^-\ominus 1^{k-r}$: simply note that there is a bijection which associates any Young diagram $\mu$ from $\lambda^-\ominus 1^{k-r}$ to the Young diagram $\mu'$ obtained by removing from $\lambda$ the $k-r$ boxes $\lambda^-\setminus\mu$ and then the bottom 1 box in each of the $n-k$ columns. The Young diagram $\mu'$ satisfies the description in \cite{BCFF} and in fact the same as $\mu$. Similarly we can see that the quantum Pieri's formula for $\sigma_\lambda*\sigma_r$ in Corollary \ref{rqpieri} is equivalent to \cite[Equation (22)]{BCFF}.
\end{remark}
Finally we will present a version of rim hook algorithm inspired by the one in \cite{BCFF}. An $n$-rim hook of a Young diagram $\lambda$ is a sequence of $n$ boxes starting with the rightmost box of a certain row and arranged along the southeast rim of $\lambda$ in the leftward and downward direction, such that what remains after deleting these boxes is still a Young diagram. The following diagram illustrates an example of a 9-rim hook in the Young diagram $(6, 3, 3, 1)$.
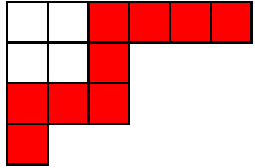
\begin{center}
	\begin{ytableau}*(white)&*(white)&*(red)&*(red)&*(red)&*(red)\\ *(white)&*(white)&*(red)\\ *(red)&*(red)&*(red)\\ *(red)
	\end{ytableau}
\end{center}

\begin{theorem}[Rim hook algorithm]\label{rimhook}
	Let $\lambda$ be a Young diagram with $\ell(\lambda)\leq k$. 
	\begin{enumerate} 
		\item If $\lambda$ has an $n$-rim hook, then
		\[\sigma_\lambda=(-1)^{k-w}q\sigma_{\lambda'}, \]
		where $w$ is the number of rows occupied by the rim hook and $\lambda'$ is the Young diagram remaining after removing the rim hook from $\lambda$. 
	\item If $\lambda_1>n-k$ and $\lambda$ does not have an $n$-rim hook, then $\sigma_\lambda=0$.
	\end{enumerate}
\end{theorem}
\begin{proof}
	Suppose $\lambda$ has an $n$-rim hook, which starts with the last box in the $\ell_1$-th row and end with a box in the $\ell_2$-th row. After deleting the $n$-rim hook, the remaining Young diagram $\lambda'$ satisfies
	\begin{enumerate}
		\item $\lambda_i'=\lambda_i$ for $1\leq i\leq \ell_1-1$ and $i\geq\ell_2+1$, 
		\item $\lambda'_i=\lambda_{i+1}-1$ for $\ell_1\leq i\leq\ell_2-1$,
		\item $\lambda'_{\ell_2}=\lambda_{\ell_1}-n+\ell_2-\ell_1\geq\lambda_{\ell_2+1}$.
	\end{enumerate}
	Consider the quantum Giambelli's formula for $\sigma_\lambda$. The $i$-th row of the relevant matrix is $(\cdots, \sigma_{\lambda_i}, \cdots, \sigma_{\lambda_i+k-i})$. The $\ell_1$-th row $(\cdots, \sigma_{\lambda_{\ell_1}}, \cdots, \sigma_{\lambda_{\ell_1}+k-\ell_1})$, by Remark \ref{rowrimhookgen}, is $(-1)^{k+1}q(\cdots, \sigma_{\lambda_{\ell_1}-n}, \cdots, \sigma_{\lambda_{\ell_1}+k-\ell_1-n})$. Moving this row to the $\ell_2$-th row and shifting the $(\ell_1+1)$st, $(\ell_1+2)$nd, $\cdots$, $\ell_2$-th rows up by one row result in $(-1)^{k+1}q\times$Giambelli's formula for $\sigma_{\lambda'}$. Overall we have
	\begin{align*}
		\sigma_\lambda&=(-1) ^{k+1}q\cdot(-1)^{\ell_2-\ell_1}\sigma_{\lambda'}\\
							&\ (\text{the factor }(-1)^{\ell_2-\ell_1}\text{ is the result of the cyclic permutation of the rows }\\
							&\ \text{from the }\ell_1\text{-th to the }\ell_2\text{-th)}\\
							&=(-1) ^{k-(\ell_2-\ell_1+1)}q\sigma_{\lambda'}\\
							&=(-1) ^{k-w}q\sigma_{\lambda'}.
	\end{align*}
	If $\lambda_1>n-k$ and $\lambda$ does not have any $n$-rim hook, then there are two possibilities. 
	\begin{enumerate}
		\item After deleting $n$ boxes along the rim starting with the rightmost box in the first row, what remains is not a Young diagram. This amounts to the equation $\lambda_{\ell_1}-n+\ell_2-\ell_1=\lambda_{\ell_2+1}-1$. Thus the $\ell_1$-th row $(-1)^{k+1}q(\cdots, \sigma_{\lambda_{\ell_1}-n}, \cdots, \sigma_{\lambda_{\ell_1}+k-\ell_1-n})$ of the Giambelli matrix is a multiple of the $(\ell_2+1)$st row. It follows that $\sigma_\lambda=0$. 
		\item There are not enough boxes along the rim, even we start with the rightmost box in the first row of $\lambda$. As the total number of boxes along the rim is $\lambda_1+\ell(\lambda)-1$, we have $\lambda_1+\ell(\lambda)-1<n$. The first row of the Giambelli matrix for $\sigma_\lambda$ is $(\sigma_{\lambda_1}, \cdots, \sigma_{\lambda_1+\ell(\lambda)-1})$, which is the zero row vector by Remark \ref{rowrimhookgen}. So again $\sigma_\lambda=0$. 
	\end{enumerate}
\end{proof}
To compute the quantum Littlewood-Richardson coefficient $c_{\lambda\mu}^\nu$ where $|\lambda|+|\mu|=|\nu|+dn$, $d>0$, we may follow the procedure below which is based on Theorem \ref{rimhook} and inspired by the rim hook algorithm in \cite{BCFF}. 
\begin{enumerate}
	\item Find the set $\Pi$ of Young diagrams $\pi$ satisfying
	\begin{enumerate}
		\item $\ell(\pi)\leq k$, 
		\item\label{b} $\pi_1-\lambda_1\leq \mu_1$, 
		\item $\lambda\subseteq\pi$, and 
		\item $\pi$ is obtained by adjoining $\nu$ with $d$ $n$-rim hooks. 
	\end{enumerate}
	Let the number of rows occupied by the $d$ $n$-rim hooks be $w_{\pi, 1}, \cdots, w_{\pi, d}$. 
	\item Find $c_{\lambda\mu}^\pi$ using the classical Littlewood-Richardson rule. Note that Condition (\ref{b}) in Step 1 just excludes the possibility that $c_{\lambda\mu}^\pi=0$ due to the reverse lattice word condition applied to the first row of the skew Young tableau $\pi\setminus\lambda$. 
	\item The quantum Littlewood-Richardson coefficient $c_{\lambda\mu}^\nu$ is the sum over $\Pi$ of signed Littlewood-Richardson coefficient $c_{\lambda\mu}^\pi$, i.e., 
	\[c_{\lambda\mu}^\nu=q^d\left(\sum_{\pi\in \Pi}(-1)^{kd-\sum_{i=1}^dw_{\pi, i}}c_{\lambda\mu}^\pi\right).\]
\end{enumerate}
\begin{remark}
	One may use the dual algebra presentation for the equivariant cohomology of Grassmannians in terms of the other set of characteristic classes $\sigma_\lambda'$ (see Definition \ref{dualchar}) and their properties in Remark \ref{dualprop} to get the dual version of Theorem \ref{rimhook}, which essentially is \cite[Main Lemma]{BCFF}, and hence the rim hook rule \emph{loc. cit}. The condition $\lambda_1>n-k$ for the vanishing of $\sigma_\lambda$ in Theorem \ref{rimhook}(2) is dual to the condition $\lambda_{k+1}>0$ (i.e. $\lambda^T_1>k$) in \cite[Main Lemma (A)]{BCFF} for the vanishing of $\sigma_\lambda$. The assumption $\ell(\lambda)\leq k$ in Theorem \ref{rimhook} is put in place to exclude those $\sigma_\lambda$ with $\ell(\lambda)>k$, which are 0 by Remark \ref{rowrimhookgen}. The condition $\ell(\lambda)>k$ is then dual to the condition $\lambda_1>n-k$ (i.e. $\ell(\lambda^T)>n-k$) in \cite[Main Lemma (A)]{BCFF} for the vanishing of $\sigma_\lambda$. The number of rows $w$ occupied by a rim hook which appears in the formula in Theorem \ref{rimhook} is dual to the width of a rim hook in apparently the same formula in \cite[Main Lemma (B)]{BCFF}. 
\end{remark}
\begin{example}(\cite[Example 1]{BCFF}). Let $X=\text{Gr}(5, 10)$, $\lambda=(5, 4, 4, 2, 2)$, $\mu=(3, 2, 1)$, $\nu=(2, 1)$. Then $d=2$, and $\Pi$ consists of only one Young diagram $\pi=(7, 7, 4, 3, 2)$. There is only one skew Young tableau in the shape $\pi\setminus\lambda$ satisfying the reverse lattice word condition, as shown below.
	
	\begin{ytableau}
		*(white)&*(white)&*(blue)&*(blue)&*(blue)&*(blue)1&*(green)1\\
		*(white)&*(blue)&*(blue)&*(green)&*(green)1&*(green)2&*(green)2\\
		*(blue)&*(blue)&*(green)&*(green)\\
		*(blue)&*(green)&*(green)3\\
		*(blue)&*(green)
	\end{ytableau}
	
	Thus $c_{\lambda\mu}^\pi=1$ and $c_{\lambda\mu}^\nu=q^2\cdot(-1)^{5+5-5-5}c_{\lambda\mu}^\pi=q^2$. The rim hook algorithm in \cite{BCFF} produces three skew Young tableaux (instead of one as in our computation), two in the shape $(5, 5, 4, 3, 2, 2, 2)$ with positive sign and one in the shape $(5, 4, 4, 3, 2, 2, 2, 1)$ with negative sign. So these signed contributions add up to 1, which is consistent with our computation. 
\end{example}
\begin{example}(\cite[Example 2]{BCFF}). Let $\lambda=(3, 3, 2, 1)$, $\mu=(4, 3, 2, 1)$ and $\nu=(4, 2, 2, 1)$. If $X=\text{Gr}(4, 10)$. Then $\Pi$ is empty and so $c_{\lambda\mu}^\nu=0$. The rim hook algorithm in \cite{BCFF} produces 8 skew Young tableaux with positive sign and 8 skew Young tableaux with negative sign, resulting in the zero signed sum, in agreement with our analysis. 

If $X=\text{Gr}(5, 10)$, then $\Pi$ consists of one Young diagram $\pi=(6, 5, 3, 3, 2)$. It turns out that $c_{\lambda\mu}^\pi=6$, as there are 6 skew Young tableaux in the shape $\pi\setminus\lambda$ satisfying the reverse lattice word condition as below.

\begin{ytableau}
	*(white)&*(white)&*(white)&*(white)1&*(green)1&*(green)1\\
	*(white)&*(white)&*(green)&*(green)2&*(green)2\\
	*(white)&*(white)&*(green)1\\
	*(white)&*(green)2&*(green)3\\
	*(green)3&*(green)4
\end{ytableau}
\begin{ytableau}
	*(white)&*(white)&*(white)&*(white)1&*(green)1&*(green)1\\
	*(white)&*(white)&*(green)&*(green)2&*(green)2\\
	*(white)&*(white)&*(green)1\\
	*(white)&*(green)3&*(green)3\\
	*(green)2&*(green)4
\end{ytableau}
\begin{ytableau}
	*(white)&*(white)&*(white)&*(white)1&*(green)1&*(green)1\\
	*(white)&*(white)&*(green)&*(green)2&*(green)2\\
	*(white)&*(white)&*(green)2\\
	*(white)&*(green)1&*(green)3\\
	*(green)3&*(green)4
\end{ytableau}

\begin{ytableau}
	*(white)&*(white)&*(white)&*(white)1&*(green)1&*(green)1\\
	*(white)&*(white)&*(green)&*(green)2&*(green)2\\
	*(white)&*(white)&*(green)2\\
	*(white)&*(green)3&*(green)3\\
	*(green)1&*(green)4
\end{ytableau}
\begin{ytableau}
	*(white)&*(white)&*(white)&*(white)1&*(green)1&*(green)1\\
	*(white)&*(white)&*(green)&*(green)2&*(green)2\\
	*(white)&*(white)&*(green)3\\
	*(white)&*(green)1&*(green)4\\
	*(green)2&*(green)3
\end{ytableau}
\begin{ytableau}
	*(white)&*(white)&*(white)&*(white)1&*(green)1&*(green)1\\
	*(white)&*(white)&*(green)&*(green)2&*(green)2\\
	*(white)&*(white)&*(green)3\\
	*(white)&*(green)2&*(green)4\\
	*(green)1&*(green)3
\end{ytableau}

So $c_{\lambda\mu}^\nu=6q$. The rim hook algorithm in \cite{BCFF} involves 8 skew Young tableaux with positive sign and 2 skew Young tableaux with negative sign, yielding a signed sum of 6, in agreement with our result. 
\end{example}
\section{Equivariant quantum Schubert calculus}\label{eqquantumSchCal}
In this section, we will obtain an algorithm for multiplication in the equivariant quantum cohomology ring of Grassmannians in terms of characteristic classes.

The equivariant quantum cohomology of flag manifolds (including Grassmannians as an example), which jointly generalizes their equivariant and quantum cohomology, was first defined and studied in \cite{GK} and \cite{Ki}. The latter paper gives an algebra presentation of the equivariant quantum cohomology ring using characteristic classes of universal quotient bundles of flag manifolds. Later in \cite{Mi} and \cite{Mi2}, rules for multiplying canonical Schubert classes in the equivariant quantum cohomology of Grassmannians were presented in the form of Pieri- and Giambelli-type formulae and a recursive algorithm for general equivariant quantum Littlewood-Richardson coefficients based on the associativity of the multiplication. We will see how these previous results, together with those in Sections \ref{equivsch} and \ref{qsch}, enable us to deduce a multiplication rule for characteristic classes in the equivariant quantum cohomology ring.

Let us use $\sigma_\lambda^\text{can}$ to denote the canonical Schubert class corresponding to the Young diagram $\lambda$ in $QH_T^*(X, \mathbb{Z})$, which by definition is the Poincar\'e dual of the homotopy quotient $ET\times_T X_\lambda$ in $ET\times_T X$. Both $\{\sigma_\lambda^\text{can}\}_{\lambda\subseteq k\times(n-k)}$ and $\{\widehat{\sigma}_\lambda\}_{\lambda\subseteq k\times(n-k)}$ are $H_T^*(\text{pt}, \mathbb{Z})[q]$-module bases for $QH_T^*(X, \mathbb{Z})$. Note that the equivariant formality of the $T$-action on $X$ implies injectivity of the restriction map $r^*: H_T^*(X, \mathbb{Z})\to H_T^*(X^T, \mathbb{Z})\cong H_T^*(\text{pt}, \mathbb{Z})^{\oplus\binom{n}{k}}$. Let us adopt the following conventions from \cite{KT}: we label the fixed points in $X^T$ by the set of 01-strings with $k$ 0's and $n-k$ 1's. To be more precise, the labeling is given by the following map.
	\begin{align*}
		\{\text{01-strings of type }(k, n-k)\}&\to X^T\\
		b&\mapsto V_b:=\text{span}\{v_i| b_i=0, 1\leq i\leq k\},
	\end{align*}
	where $\{v_i\}_{i=1}^n$ is the standard ordered basis of $\mathbb{C}^n$. One can further label a 01-string with a Young diagram in the $k\times(n-k)$ rectangle: a Young diagram $\lambda$ corresponds to the 01-string obtained by labeling each vertical segment and horizontal segment of the southeast boundary of $\lambda$ by 0 and 1 respectively, and reading the labels from the southwest corner to the northeast corner. Let $b(\lambda)$ be the 01-string associated with $\lambda$, and $\lambda(b)$ the Young diagram associated with $b$. With these labels, one can define the Bruhat order on $X^T$ and 01-strings. We say $V_b>V_c$ and $b>c$ if $\lambda(b)$ includes $\lambda(c)$ as a subdiagram. Let $\sigma$ be an equivariant cohomology class in $H_T^*(X, \mathbb{Z})$. Define $(r^*(\sigma))_\mu$ to be the restriction of $\sigma$ to the fixed point $V_{b(\mu)}$. We say $\sigma$ is supported above $\mu$ if $(r^*(\sigma))_{\mu'}\neq 0$ implies $\mu\subseteq\mu'$. There is a useful characterization of canonical Schubert classes (cf. \cite[Lemma 1]{KT} and the preceding discussions) in terms of their restriction to $X^T$: $\sigma_\lambda^\text{can}$ is the unique equivariant cohomology class satisfying
	\begin{enumerate}
		\item\label{canvanish} $\sigma_\lambda^\text{can}$ is supported above $\lambda$, i.e., $(r^*(\sigma_\lambda^\text{can}))_\mu:=$restriction of $\sigma_\lambda^\text{can}=0$ for any $\mu\subset\lambda$ or $\mu$ and $\lambda$ not comparable, 
		\item\label{minrestriction} $\displaystyle(r^*(\sigma_\lambda^\text{can}))_\lambda=\prod_{\substack{1\leq i<j\leq n\\ b(\lambda)_i=1, b(\lambda)_j=0}}(t_i-t_j)$,
		\item $(r^*(\sigma_\lambda^\text{can}))_\mu$ is a homogeneous polynomial in $t_1, \cdots, t_n$ of degree $|\lambda|$, and
		\item(the GKM condition) if $b(\mu_1)$ and $b(\mu_2)$ only differ in the $i$-th and $j$-th entries, then $(r^*(\sigma_\lambda^\text{can}))_{\mu_1}-(r^*(\sigma_\lambda^\text{can}))_{\mu_2}$ is divisible by $t_i-t_j$. 
	\end{enumerate}\begin{proposition}\label{cantochar}
	In $QH_T^*(X, \mathbb{Z})$, for $1\leq r\leq k$, 
	\[\widehat{\sigma}_{1^r}=\sum_{i=0}^r(-1)^{r-i}e_{r-i}(k-i)\sigma_{1^{i}}^\text{can},\]
	where $e_i(m)$ is the $i$-th elementary symmetric polynomial in the $m$ equivariant variables $t_1, \cdots, t_m$.
\end{proposition} 
\begin{proof}
	As the set of canonical Schubert classes is a $H_T^*(\text{pt}, \mathbb{Z})[q]$-module basis for $QH_T^*(X, \mathbb{Z})$, $\widehat{\sigma}_{1^r}$ can be expressed as a $H_T^*(\text{pt}, \mathbb{Z})[q]$-linear combination of canonical Schubert classes. This expression does not involve quantum deformation because the degree of $q$, which is $n$, is greater than that of $\widehat{\sigma}_{1^r}$, which is $r$. 
	
	Note that the injectivity of the restriction map $r^*$ implied by the equivariant formality of the $T$-action allows us to prove the desired equation at the level of the restriction of equivariant Schubert classes to fixed points. By the definition of $\widehat{\sigma}_{1^r}$ as $c_r(S^*)$, we have that 
	\[(r^*(\widehat{\sigma}_{1^r}))_\mu=(-1)^re_r(t_{i_1}, t_{i_2}, \cdots, t_{i_k})\]
	where $b(\mu)_{i_j}=0$ for $1\leq j\leq k$ and $e_r(t_{i_1}, t_{i_2}, \cdots, t_{i_k})$ is the $r$-th elementary symmetric polynomial in $t_{i_1}, t_{i_2}, \cdots, t_{i_k}$. To prove that $\widehat{\sigma}_{1^r}$ is the linear combination of $\sigma_{1^i}^\text{can}$ for $0\leq i\leq r$ given in the Proposition, we shall use the `upper triangularity' of the canonical Schubert classes with respect to the Bruhat order of $X^T$ and follow the algorithm of expressing a general cohomology class as a linear combination of canonical Schubert classes as outlined in the proof of \cite[Proposition 1]{KT}: let $\sigma$ be an equivariant cohomology class which is supported above $\mu$. Then $(r^*(\sigma))_\mu=\beta\cdot(r^*(\sigma_\mu^\text{can}))_\mu$ for some $\beta\in H_T^*(\text{pt}, \mathbb{Z})$. Then $\sigma-\beta\cdot\sigma_\mu^\text{can}$ is supported above a certain Young diagram $\mu'$ which includes $\mu$. Repeat the same procedure to $\sigma-\beta\cdot\sigma_\mu^\text{can}$ until we get 0. We find the following useful in carrying out this algorithm to our problem at hand.
	\begin{claim}\label{claim}
		For $j\leq r-1$, 
		\begin{align*}
			&\left(r^*\left(\widehat{\sigma}_{1^r}-\sum_{i=0}^j(-1)^{r-i}e_{r-i}(k-i)\sigma_{1^i}^\text{can}\right)\right)_{1^\ell}\\
			=&\begin{cases}\displaystyle(-1)^{r-j-1}e_{r-j-1}(k-j-1)\prod_{p=1}^{j+1}(t_{k-j}-t_{k-j+p})&,\ \text{if }\ell=j+1\\ 0&,\ \text{if }\ell\leq j\end{cases}.
		\end{align*}
	\end{claim}
	We shall prove by induction on $j$. The case $j=0$ is easy and left to the reader. Assume that the claim is true for $j$. We shall show that
	\begin{align*}
			&\left(r^*\left(\widehat{\sigma}_{1^r}-\sum_{i=0}^{j+1}(-1)^{r-i}e_{r-i}(k-i)\sigma_{1^i}^\text{can}\right)\right)_{1^\ell}\\
			=&\begin{cases}\displaystyle(-1)^{r-j-2}e_{r-j-2}(k-j-2)\prod_{p=1}^{j+2}(t_{k-j-1}-t_{k-j-1+p})&,\ \text{if }\ell=j+2\\ 0&,\ \text{if }\ell\leq j+1\end{cases}.
		\end{align*}
	Note that for $\ell\leq j$, 
	\begin{align*}
		&\left(r^*\left(\widehat{\sigma}_{1^r}-\sum_{i=0}^{j+1}(-1)^{r-i}e_{r-i}(k-i)\sigma_{1^i}^\text{can}\right)\right)_{1^\ell}\\
		=&\left(r^*\left(\widehat{\sigma}_{1^r}-\sum_{i=0}^{j}(-1)^{r-i}e_{r-i}(k-i)\sigma_{1^i}^\text{can}\right)\right)_{1^\ell}-(r^*((-1)^{r-j-1}e_{r-j-1}(k-j-1)\sigma_{1^{j+1}}^\text{can}))_{1^\ell}\\
		=&0-0\ (\text{by induction hypothesis and (\ref{canvanish}) of characterization of canonical Schubert classes})\\
		=&0.
	\end{align*}
	For $\ell=j+1$, 
	\begin{align*}
		&\left(r^*\left(\widehat{\sigma}_{1^r}-\sum_{i=0}^{j}(-1)^{r-i}e_{r-i}(k-i)\sigma_{1^i}^\text{can}\right)\right)_{1^\ell}\\
		=&(-1)^{r-j-1}e_{r-j-1}(k-j-1)\prod_{p=1}^{j+1}(t_{k-j}-t_{k-j+p})\ (\text{by induction hypothesis})\\
		=&(r^*((-1)^{r-j-1}e_{r-j-1}(k-j-1)\sigma_{1^{j+1}}^\text{can}))_{1^{j+1}}\\
		& (\text{by (\ref{minrestriction}) of characterization of canonical Schubert classes}).
	\end{align*}
	So $\displaystyle\left(r^*\left(\widehat{\sigma}_{1^r}-\sum_{i=0}^{j+1}(-1)^{r-i}e_{r-i}(k-i)\sigma_{1^i}^\text{can}\right)\right)_{1^\ell}=0$ for $\ell\leq j+1$. It remains to show that 
	\begin{align*}
		&\left(r^*\left(\widehat{\sigma}_{1^r}-\sum_{i=0}^{j+1}(-1)^{r-i}e_{r-i}(k-i)\sigma_{1^i}^\text{can}\right)\right)_{1^{j+2}}\\
		=&(-1)^{r-j-2}e_{r-j-2}(k-j-2)\prod_{p=1}^{j+2}(t_{k-j-1}-t_{k-j-1+p}).
	\end{align*}
	Now we shall use the following observation: let $\sigma\in H_T^*(X, \mathbb{Z})$ be an equivariant Schubert class which is a characteristic class, or a canonical Schubert class, $b_1$ and $b_2$ two 01-strings which differ only in the $i$-th and the $j$-th entries. If $(r^*(\sigma))_{\lambda(b_1)}$ involves the variable $t_i$ or $t_j$, then $(r^*(\sigma))_{\lambda(b_2)}$ is $(r^*(\sigma))_{\lambda(b_1)}$ with $t_i$ and $t_j$ swapped. Note that the 01-strings associated with $1^{j+1}$ and $1^{j+2}$ differ only in the $(k-j-1)$-st and the $(k-j)$-th entries. Applying this observation to the following inductive hypothesis
	\begin{align*}
		&\left(r^*\left(\widehat{\sigma}_{1^r}-\sum_{i=0}^{j}(-1)^{r-i}e_{r-i}(k-i)\sigma_{1^i}^\text{can}\right)\right)_{1^{j+1}}\\
		=&(-1)^{r-j-1}e_{r-j-1}(k-j-1)\prod_{p=1}^{j+1}(t_{k-j}-t_{k-j+p}), 
	\end{align*}
	and bearing in mind that $e_{r-i}(k-i)$ is invariant under the swapping of $t_{k-j-1}$ and $t_{k-j}$ for $i\leq j$, we have
	\begin{align*}
		&\left(r^*\left(\widehat{\sigma}_{1^r}-\sum_{i=0}^{j}(-1)^{r-i}e_{r-i}(k-i)\sigma_{1^i}^\text{can}\right)\right)_{1^{j+2}}\\
		=&(-1)^{r-j-1}e_{r-j-1}(t_1, t_2, \cdots, t_{k-j-2}, t_{k-j})\prod_{p=1}^{j+1}(t_{k-j-1}-t_{k-j+p}).
	\end{align*}
	Applying the same observation again to 
	\begin{align*}
		&(r^*((-1)^{r-j-1}e_{r-j-1}(k-j-1)\sigma_{1^{j+1}}^\text{can}))_{1^{j+1}}\\
		=&(-1)^{r-j-1}e_{r-j-1}(k-j-1)\prod_{p=1}^{j+1}(t_{k-j}-t_{k-j+p}), 
	\end{align*}
	we have
	\begin{align*}
		&(r^*((-1)^{r-j-1}e_{r-j-1}(k-j-1)\sigma_{1^{j+1}}^\text{can}))_{1^{j+2}}\\
		=&(-1)^{r-j-1}e_{r-j-1}(k-j-1)\prod_{p=1}^{j+1}(t_{k-j-1}-t_{k-j+p}).
	\end{align*}
	It follows that 
	\begin{align*}
		&\left(r^*\left(\widehat{\sigma}_{1^r}-\sum_{i=0}^{j+1}(-1)^{r-i}e_{r-i}(k-i)\sigma_{1^i}^\text{can}\right)\right)_{1^{j+2}}\\
		=&(-1)^{r-j-1}(e_{r-j-1}(t_1, t_2, \cdots, t_{k-j-2}, t_{k-j})-e_{r-j-1}(k-j-1))\prod_{p=1}^{j+1}(t_{k-j-1}-t_{k-j+p})\\
		=&(-1)^{r-j-1}e_{r-j-2}(k-j-2)(t_{k-j}-t_{k-j-1})\prod_{p=1}^{j+1}(t_{k-j-1}-t_{k-j+p})\\
		=&(-1)^{r-j-2}e_{r-j-2}(k-j-2)\prod_{p=1}^{j+2}(t_{k-j-1}-t_{k-j-1+p}).
	\end{align*}
	We have established Claim \ref{claim}. In particular, the case $j=r-1$ says that $\displaystyle \widehat{\sigma}_{1^r}-\sum_{i=0}^{r-1}(-1)^{r-i}e_{r-i}(k-i)\sigma_{1^i}^\text{can}$ is supported above $1^r$, and 
	\[\left( r^*\left(\widehat{\sigma}_{1^r}-\sum_{i=0}^{r-1}(-1)^{r-i}e_{r-i}(k-i)\sigma_{1^i}^\text{can}\right)\right)_{1^r}=\prod_{p=1}^r(t_{k-r+1}-t_{k-r+1+p})=(r^*(\sigma_{1^r}^\text{can}))_{1^r}.\]
	Thus $\widehat{\sigma}_{1^r}-\sum_{i=0}^{r}(-1)^{r-i}e_{r-i}(k-i)\sigma_{1^i}^\text{can}=0$, for otherwise it would be supported above a certain $\mu'$ which includes $1^r$ and so by the algorithm contain the term $\beta\cdot\sigma_{\mu'}^\text{can}$ for some nonzero $\beta\in H_T^*(\text{pt}, \mathbb{Z})$ in its expression as the linear combination of canonical Schubert classes, whose degree is greater than that of $\widehat{\sigma}_{1^r}-\sum_{i=0}^{r}(-1)^{r-i}e_{r-i}(k-i)\sigma_{1^i}^\text{can}$, a contradiction. This finishes the proof of the Proposition.
\end{proof}
\begin{proposition}\label{noquantdef}
	Let $\lambda$ and $\mu$ be two Young diagrams satisfying $\lambda_1+\mu_1\leq n-k$. Then in $QH_T^*(X, \mathbb{Z})$, the product $\sigma_\lambda^\text{can}*\sigma_\mu^\text{can}$ does not involve quantum deformation, and is a linear combination of classes $\sigma_\nu^\text{can}$ where $\nu_1\leq \lambda_1+\mu_1$.
\end{proposition}
\begin{proof}
	By Theorem \cite[Theorem 1.1(b)]{Mi2}, the canonical Schubert classes $\sigma_\lambda^\text{can}$ can be represented by factorial Schur polynomials $\widetilde{s}_\lambda$ (see \cite[\S 1.1]{Mi2} for definition), and $QH_T^*(X, \mathbb{Z})$ is isomorphic to the ring 
	\[\frac{\mathbb{Z}[\widetilde{s}_1, \widetilde{s}_{1^2}, \cdots, \widetilde{s}_{1^k}, t_1, \cdots, t_n, q]}{(H_{n-k+1}, H_{n-k+2}, \cdots, H_n+(-1)^kq)}\]
	via the map which sends $\widetilde{s}_\lambda$ to $\sigma_\lambda^\text{can}$ (here $H_i$ is the factorial version of Giambelli's formula for $\widetilde{s}_i$ in terms of factorial Schur polynomials corresponding to column Young diagrams). The factorial Schur polynomials also represent canonical Schubert classes in $H_T^*(X, \mathbb{Z})$ which is isomorphic to almost the same ring except that the last generator in the relation ideal does not have quantum deformation (cf. \cite[Corollary 5.1(b)]{Mi2}):
	\[H_T^*(X, \mathbb{Z})\cong\frac{\mathbb{Z}[\widetilde{s}_1, \widetilde{s}_{1^2}, \cdots, \widetilde{s}_{1^k}, t_1, \cdots, t_n, q]}{(H_{n-k+1}, H_{n-k+2}, \cdots, H_n)}.\]
	By the factorial version of Littlewood-Richardson's rule (cf. \cite{MS}), $\widetilde{s}_\lambda\cdot\widetilde{s}_\mu$ is a linear combination of $\widetilde{s}_\nu$ where $\nu_1\leq \lambda_1+\mu_1$ with coefficients involving only the equivariant variables (quantum deformation only appears if the linear combination involves some $\widetilde{s}_\nu$ with $\nu_1>n-k$, which can be written as a linear combination of $\widetilde{s}_{\nu'}$ with all $\nu'\subseteq k\times(n-k)$ with coefficients involving $q$, by factorial Giambelli's formula and the last generator $H_n+(-1)^kq$ of the relation ideal of $QH_T^*(X, \mathbb{Z})$).
\end{proof}
\begin{proof}[Proof of Theorem \ref{eqquantgiam} (equivariant quantum Giambelli's formula)]
	The equivariant Giambelli's formula $\widehat{\sigma}_\lambda=\det(\widehat{\sigma}_{1^{\lambda_i^T+j-i}})_{1\leq i, j\leq\ell(\lambda^T)}=\det(\widehat{\sigma}_{\lambda_i+j-i})_{1\leq i, j\leq\ell(\lambda)}$ holds in $H_{U(n)}^*(X, \mathbb{Z})$ by Definition \ref{equivGiambelli} and Remark \ref{coincide}. By Propositions \ref{cantochar} and \ref{noquantdef}, there is no quantum deformation in $\det(\widehat{\sigma}_{1^{\lambda_i^T+j-i}})_{1\leq i, j\leq \ell(\lambda^T)}$ and thus $\widehat{\sigma}_\lambda=\det(\widehat{\sigma}_{1^{\lambda_i^T+j-i}})_{1\leq i, j\leq \ell(\lambda^T)}$ still holds in $QH_{U(n)}^*(X, \mathbb{Z})$. Thus $\sigma_\lambda^\text{can}*\sigma_\mu^\text{can}\in QH_T^*(X, \mathbb{Z})$ and $\sigma_\lambda^\text{can}\cdot\sigma_\mu^\text{can}\in H_T^*(X, \mathbb{Z})$ are the same and both do not involve quantum deformation. 
\end{proof}
\begin{proposition}\label{eqquantumhomo}
	The map 
	\[\widetilde{f}: H_{U(n)}^*(X, \mathbb{Z})\to QH_{U(n)}^*(X, \mathbb{Z})\]
	which 
	\begin{enumerate}
		\item acts as the identity on the generators $\{\widehat{\sigma}_i| 1\leq i\leq n-k\}$ and $\{e_i| 1\leq i\leq n-1\}$, and 
		\item sends $e_n$ to $e_n+(-1)^kq$
	\end{enumerate}
	defines a ring homomorphism.
\end{proposition}
\begin{proof}
	By Theorem 2 of \cite{Ki}, we have
	\[QH_{U(n)}^*(X, \mathbb{Z})\cong\frac{\mathbb{Z}[\widehat{\sigma}_1, \cdots, \widehat{\sigma}_{n-r}, e_1, \cdots, e_n]}{(\widehat{Y}_{k+1}, \widehat{Y}_{k+2}, \cdots, \widehat{Y}_n+(-1)^{n-k}q)}.\]
	Comparing this with the presentation of $H_{U(n)}^*(X, \mathbb{Z})$ at the beginning of Section \ref{linking} which has the same set of generators, one can see that $\widetilde{f}$ maps the relation ideal of $H_{U(n)}^*(X, \mathbb{Z})$ to that of $QH_{U(n)}^*(X, \mathbb{Z})$. Thus $\widetilde{f}$ indeed defines a ring homormorphism from $H_{U(n)}^*(X, \mathbb{Z})$ to $QH_{U(n)}^*(X, \mathbb{Z})$. 
\end{proof}
\begin{proof}[Proof of Theorem \ref{equivequivquantum}]
	By Theorem \ref{eqquantgiam} and Proposition \ref{eqquantumhomo}, $\widetilde{f}(\widehat{\sigma}_\lambda)=\widehat{\sigma}_\lambda$ for any $\lambda\subseteq k\times(n-k)$. The Theorem then immediately follows from Proposition \ref{equivequivquantum}.
\end{proof}
\begin{proof}[Proof of Theorem \ref{equivquantpieri}]
This follows from Theorem \ref{equivpieri} and Theorem \ref{equivequivquantum}. 
\end{proof}
\begin{remark}
	Like the equivariant quantum Pieri's formula in \cite{Mi} in terms of canonical Schubert classes, ours also do not have any `mixed' term, i.e. the term which contains both quantum and equivariant variables.
\end{remark}
\begin{example}
Let $X=\text{Gr}(2, 4)$. To compute the equivariant quantum product $\widehat{\sigma}_{22}*\widehat{\sigma}_{21}\in QH_{U(n)}^*(X, \mathbb{Z})$, we first compute the equivariant product $\widehat{\sigma}_{22}\cdot\widehat{\sigma}_{21}\in H_{U(n)}^*(X, \mathbb{Z})$ using classical Littlewood-Richardson's rule and Giambelli's formula, as well as Proposition \ref{ringstr} (\ref{simplify}). 
\begin{align*}
	&\widehat{\sigma}_{22}\cdot\widehat{\sigma}_{21}\\
	=&\widehat{\sigma}_{43}\\
	=&\left|\begin{matrix}\widehat{\sigma}_4&\widehat{\sigma}_5\\ \widehat{\sigma}_2&\widehat{\sigma}_3\end{matrix}\right|\\
	=&\left|\begin{matrix}-e_4-e_3\widehat{\sigma}_1-e_2\widehat{\sigma}_2-e_1\widehat{\sigma}_3& -e_4\widehat{\sigma}_1-e_3\widehat{\sigma}_2-e_2\widehat{\sigma}_3-e_1\widehat{\sigma}_4\\ \widehat{\sigma}_2&\widehat{\sigma}_3\end{matrix}\right|\\
	=&-e_4\left|\begin{matrix}1& \widehat{\sigma}_1\\ \widehat{\sigma}_2&\widehat{\sigma}_3\end{matrix}\right|-e_3\left|\begin{matrix}\widehat{\sigma}_1&\widehat{\sigma}_2\\ \widehat{\sigma}_2&\widehat{\sigma}_3\end{matrix}\right|-e_2\left|\begin{matrix}\widehat{\sigma}_2&\widehat{\sigma}_2\\ \widehat{\sigma}_2&\widehat{\sigma}_3\end{matrix}\right|-e_1\left|\begin{matrix}\widehat{\sigma}_3&\widehat{\sigma}_4\\ \widehat{\sigma}_2&\widehat{\sigma}_3\end{matrix}\right|\\
	=&e_4\widehat{\sigma}_{21}+e_3\widehat{\sigma}_{22}-e_1\left|\begin{matrix}-e_3-e_2\widehat{\sigma}_1-e_1\widehat{\sigma}_2& -e_4-e_3\widehat{\sigma}_1-e_2\widehat{\sigma}_2-e_1\widehat{\sigma}_3\\ \widehat{\sigma}_2&\widehat{\sigma}_3\end{matrix}\right|\\
	=&e_4\widehat{\sigma}_{21}+e_3\widehat{\sigma}_{22}-e_1\left(-e_4\left|\begin{matrix}0&1\\ \widehat{\sigma}_2&\widehat{\sigma}_3\end{matrix}\right|-e_3\left|\begin{matrix}1& \widehat{\sigma}_1\\ \widehat{\sigma}_2&\widehat{\sigma}_3\end{matrix}\right|-e_2\left|\begin{matrix}\widehat{\sigma}_1&\widehat{\sigma}_2\\ \widehat{\sigma}_2&\widehat{\sigma}_3\end{matrix}\right|-e_1\left|\begin{matrix}\widehat{\sigma}_2&\widehat{\sigma}_3\\ \widehat{\sigma}_2&\widehat{\sigma}_3\end{matrix}\right|\right)\\
	=&e_4\widehat{\sigma}_{21}+e_3\widehat{\sigma}_{22}-e_1(e_4\widehat{\sigma}_2+e_3\widehat{\sigma}_{21}+e_2\widehat{\sigma}_{22})\\
	=&(e_3-e_1e_2)\widehat{\sigma}_{22}+(e_4-e_1e_3)\widehat{\sigma}_{21}-e_1e_4\widehat{\sigma}_2.
\end{align*} 
By Theorem \ref{equivequivquantum}, we have 
\[\widehat{\sigma}_{22}*\widehat{\sigma}_{21}=(e_3-e_1e_2)\widehat{\sigma}_{22}+(e_4+q-e_1e_3)\widehat{\sigma}_{21}-e_1(e_4+q)\widehat{\sigma}_2.\]
Below is the multiplication table for $QH_T^*(\text{Gr}(2, 4), \mathbb{Z})$ with respect to the basis of characteristic classes.
\end{example}
\begin{tiny}
	\begin{center}
		\begin{tabular}{|c|c|c|c|c|c|}
			\hline
			& $\widehat{\sigma}_1$&$\widehat{\sigma}_2$&$\widehat{\sigma}_{11}$&$\widehat{\sigma}_{21}$&$\widehat{\sigma}_{22}$\\
			\hline
			$\widehat{\sigma}_1$&$\widehat{\sigma}_1+\widehat{\sigma}_{11}$&\begin{tabular}{c}$\widehat{\sigma}_{21}-e_1\widehat{\sigma}_2$\\$-e_2\widehat{\sigma}_1-e_3$\end{tabular}&$\widehat{\sigma}_{21}$&\begin{tabular}{c}$\widehat{\sigma}_{2, 2}-e_1\widehat{\sigma}_{21}$\\ $-e_2\widehat{\sigma}_{11}+e_4+q$\end{tabular}&\begin{tabular}{c}$-e_1\widehat{\sigma}_{22}+e_3\widehat{\sigma}_{11}$\\$+(e_4+q)\widehat{\sigma}_1$\end{tabular}\\
			\hline
			$\widehat{\sigma}_2$& &\begin{tabular}{c}$\widehat{\sigma}_{22}-e_1\widehat{\sigma}_{21}$\\ $-e_2\widehat{\sigma}_{11}+(e_1^2-e_2)\widehat{\sigma}_2$\\$+(e_1e_2-e_3)\widehat{\sigma}_1+e_1e_3$\end{tabular}&\begin{tabular}{c}$-e_1\widehat{\sigma}_{21}-e_2\widehat{\sigma}_{11}$\\ $+e_4+q$\end{tabular}&\begin{tabular}{c}$-e_1\widehat{\sigma}_{22}+(e_1^2-e_2)\widehat{\sigma}_{21}$\\ $e_1e_2\widehat{\sigma}_{11}$\end{tabular}& \begin{tabular}{c}$(e_1^2-e_2)\widehat{\sigma}_{22}$\\$+(-e_1e_3+e_4+q)\widehat{\sigma}_{11}$\\ $-e_1(e_4+q)\widehat{\sigma}_1$\end{tabular}\\
			\hline
			$\widehat{\sigma}_{11}$& &&$\widehat{\sigma}_{22}$&\begin{tabular}{c}$-e_1\widehat{\sigma}_{22}+e_3\widehat{\sigma}_{11}$\\ $+(e_4+q)\widehat{\sigma}_1$\end{tabular}&\begin{tabular}{c}$e_2\widehat{\sigma}_{22}+e_3\widehat{\sigma}_{21}$\\ $+(e_4+q)\widehat{\sigma}_2$\end{tabular}\\
			\hline
			$\widehat{\sigma}_{21}$& &&&\begin{tabular}{c}$e_1^2\widehat{\sigma}_{22}+e_3\widehat{\sigma}_{21}$\\ $+(e_4+q-e_1e_3)\widehat{\sigma}_{11}$\\ $+(e_4+q)\widehat{\sigma}_2$\\ $-e_1(e_4+q)\widehat{\sigma}_1$\end{tabular}&\begin{tabular}{c}$(e_3-e_1e_2)\widehat{\sigma}_{22}$\\ $+(e_4+q-e_1e_3)\widehat{\sigma}_{21}$\\$-e_1(e_4+q)\widehat{\sigma}_2$\end{tabular}\\
			\hline
			$\widehat{\sigma}_{22}$& &&&& \begin{tabular}{c}$(e_2^2-e_1e_3)\widehat{\sigma}_{22}$\\ $+(e_2e_3-e_1(e_4+q))\widehat{\sigma}_{21}$\\ $(e_3^2-e_2(e_4+q))\widehat{\sigma}_{11}$\\$+e_2(e_4+q)\widehat{\sigma}_2$\\$+e_3(e_4+q)\widehat{\sigma}_1$\\$+(e_4+q)^2$\end{tabular} \\
			\hline
		\end{tabular}
	\end{center}
\end{tiny}

\noindent\footnotesize{\textsc{NYU-ECNU Institute of Mathematical Sciences, \\
New York University Shanghai, \\
3663 Zhongshan Road North,\\
Shanghai 200062, China}\\
\\
\textsc{E-mail}: \texttt{ckfok@nyu.edu}\\
\textsc{URL}: \texttt{https://sites.google.com/site/alexckfok}

\begin{thebibliography}{9999}



\bibitem[BCFF]{BCFF} 
A. Bertram, I. Ciocan-Fontainine, W. Fulton, 
\emph{Quantum multiplication of Schur polynomials}, 
J. Algebra, 219 (2): 728-746, 1999. 


	

\bibitem[Ber]{Ber}
A. Bertram, 
\emph{Quantum Schubert Calculus}, 
Advances in Math., 128, 289-305, 1997.

\bibitem[BT]{BT}
R. Bott, L. Tu, 
\emph{Differential forms in algebraic topology}, 
Graduate Texts in Mathematics 82, Springer Verlag, New York, 1982.

\bibitem[Bu]{Bu}
A. S. Buch, 
\emph{Quantum cohomology of Grassmannians}, 
Compositio Mathematica Vol.137, pp. 227--235, 2003.

\bibitem[FH]{FH} 
W. Fulton, J. Harris, 
\emph{Representation theory, a first course}, 
Graduate Text in Mathematics, Vol. 129, Springer-Verlag, 1991.





\bibitem[GK]{GK}
A. Givental, B. Kim, 
\emph{Quantum cohomology of flag manifolds and Toda lattices}, 
Commun. Math. Phys. 168(3), pp. 609--641, 1995.

\bibitem[GKM]{GKM} 
M. Goresky, R. Kottwitz, R. Macpherson, 
\emph{Equivariant cohomology, Koszul duality, and the localization theorem}, 
Invent. Math. 131, no. 1, 25-83, 1998. 



\bibitem[Ki]{Ki}
B. Kim, 
\emph{On equivariant quantum cohomology}, 
Int. Math. Res. Not., 17, 841-851, 1996.

\bibitem[KM]{KM}
M. Kontsevich, Yu. Manin, 
\emph{Gromov-Witten classes, quantum cohomology, and enumerative geometry}, 
Comm. Math. Phys. 164(3): 525--562, 1994.

\bibitem[KT]{KT} 
A. Knutson, T. Tao, 
\emph{Puzzles and (equivariant) cohomology of Grassmannians}, 
Duke Math. J., 119(2): 221-260, 2003.

\bibitem[Mi]{Mi} 
L. C. Mihalcea, \emph{Equivariant quantum Schubert calculus}, 
Adv. Math., 203(1): 1-33, 2006.

\bibitem[Mi2]{Mi2} 
L. C. Mihalcea, \emph{Giambelli formulae for the equivariant quantum cohomology of the Grassmannian}, 
Trans. Amer. Math. Soc., Vol. 360, no. 5, pp. 2285--2301, 2008.

\bibitem[MS]{MS}
A. I. Molev, B. Sagan, 
\emph{A Littlewood-Richardson rule for factorial Schur functions}, 
Trans. of Amer. Math. Soc., 351(11):4429–4443, 1999.

\bibitem[RT]{RT}
Y. Ruan, G. Tian, 
\emph{A mathematical theory of quantum cohomology}, 
J. Differential Geometry, Vol. 42, No. 2, pp. 259--367, 1995.

\bibitem[ST]{ST}
B. Siebert, G. Tian, 
\emph{On quantum cohomology rings of Fano manifolds and a formula of Vafa and Intriligator}, 
Asian J. Math. Vol. 1, No. 4, pp. 679--695, Dec. 1997.

\bibitem[W]{W} 
E. Witten, 
\emph{The Verlinde algebra and the cohomology of the Grassmannian}, 
preprint, \texttt{hep-th 9312104.}

\end{thebibliography}
\end{document}